\documentclass[12pt]{amsart}
\usepackage[margin=1in]{geometry}
\usepackage{amssymb,amsfonts,amsmath}
\usepackage{color}
\usepackage{soul}
\usepackage{enumerate}
\usepackage{mathrsfs}
\usepackage[unicode]{hyperref}
\hypersetup{
	pdftitle={Towards a \texorpdfstring{$\mathrm{GL}_n$}{GL\9040\231} variant of the Hoheisel phenomenon},
	pdfauthor={Peter Humphries and Jesse Thorner}
}

\usepackage{constants}
\usepackage[capitalise]{cleveref}

\newtheorem{theorem}{Theorem}[section]
\newtheorem{corollary}[theorem]{Corollary}

\newtheorem{proposition}[theorem]{Proposition}
\newtheorem{lemma}[theorem]{Lemma}

\newtheorem{question*}{Question}
\newtheorem{problem*}{Problem}

\theoremstyle{definition}

\theoremstyle{remark}
\newtheorem*{remark}{Remark}

\numberwithin{equation}{section}

\crefname{figure}{Figure}{Figures}
\theoremstyle{plain}
\newtheorem*{theorem*}{Theorem}
\crefname{theorems}{Theorem}{Theorems}
\crefname{corollaries}{Corollary}{Corollaries}
\newtheorem*{corollary*}{Corollary}
\crefname{corollaries*}{Corollary}{Corollaries}
\crefname{lemma}{Lemma}{Lemmas}
\crefname{proposition}{Proposition}{Propositions}
\crefname{conjectures}{Conjecture}{Conjectures}
\newtheorem*{conjonjecture*}{Conjecture}
\crefname{conjonjectures*}{Conjecture}{Conjectures}
\crefname{definitions}{Definition}{Definitions}
\crefname{hypotheses}{Hypothesis}{Hypotheses}

\newcommand{\R}{\mathbb{R}}
\newcommand{\Q}{\mathbb{Q}}

\newcommand{\kn}{\mathfrak{n}}
\newcommand{\kd}{\mathfrak{d}}
\newcommand{\kp}{\mathfrak{p}}

\newcommand{\re}{\textup{Re}}
\newcommand{\im}{\textup{Im}}

\allowdisplaybreaks

\newcommand{\GL}{\mathrm{GL}}

\newcommand{\N}{\mathrm{N}}
\newcommand{\A}{\mathbb{A}}

\makeatletter
\DeclareFontFamily{U}  {MnSymbolF}{}
\DeclareSymbolFont{symbolsMN}{U}{MnSymbolF}{m}{n}
\SetSymbolFont{symbolsMN}{bold}{U}{MnSymbolF}{b}{n}
\DeclareFontShape{U}{MnSymbolF}{m}{n}{
    <-6>  MnSymbolF5
   <6-7>  MnSymbolF6
   <7-8>  MnSymbolF7
   <8-9>  MnSymbolF8
   <9-10> MnSymbolF9
  <10-12> MnSymbolF10
  <12->   MnSymbolF12}{}
\DeclareFontShape{U}{MnSymbolF}{b}{n}{
    <-6>  MnSymbolF-Bold5
   <6-7>  MnSymbolF-Bold6
   <7-8>  MnSymbolF-Bold7
   <8-9>  MnSymbolF-Bold8
   <9-10> MnSymbolF-Bold9
  <10-12> MnSymbolF-Bold10
  <12->   MnSymbolF-Bold12}{}
\DeclareMathSymbol{\tbigtimes}{\mathop}{symbolsMN}{2}
\newcommand*{\bigtimes}{%
  \DOTSB
  \tbigtimes
  \slimits@ 
}
\makeatother

\renewcommand{\tilde}{\widetilde}

\renewcommand{\epsilon}{\varepsilon}

\newcommand{\cO}{\mathcal{O}}

\newcommand{\kq}{\mathfrak{q}}

\newconstantfamily{abcon}{symbol=c}

\makeatletter
\let\@wraptoccontribs\wraptoccontribs
\makeatother

\title[]{Towards a $\mathrm{GL}_n$ variant of the Hoheisel phenomenon}
\author{Peter Humphries}
\address{Department of Mathematics, University of Virginia, Charlottesville, VA 22904, USA}
\email{\href{mailto:pclhumphries@gmail.com}{pclhumphries@gmail.com}}
\urladdr{\href{https://sites.google.com/view/peterhumphries/}{https://sites.google.com/view/peterhumphries/}}
\author{Jesse Thorner}
\address{Department of Mathematics, University of Illinois, Urbana, IL 61801, USA}
\email{\href{mailto:jesse.thorner@gmail.com}{jesse.thorner@gmail.com}}

\begin{document}

\begin{abstract}
Let $\pi$ be a unitary cuspidal automorphic representation of $\GL_n$ over a number field, and let $\tilde{\pi}$ be contragredient to $\pi$.  We prove effective upper and lower bounds of the correct order in the short interval prime number theorem for the Rankin--Selberg $L$-function $L(s,\pi\times\tilde{\pi})$, extending the work of Hoheisel and Linnik.  Along the way, we prove for the first time that $L(s,\pi\times\widetilde{\pi})$ has an unconditional standard zero-free region apart from a possible Landau--Siegel zero.
\end{abstract}

\thanks{The first author is supported by the European Research Council grant agreement 670239.}

\maketitle

\section{Introduction}

It is well-known that if there exists a constant $0<\delta<\frac{1}{2}$ such that the Riemann zeta function $\zeta(s)$ is nonzero in the region $\re(s)\geq 1-\delta$, then the primes are regularly distributed in intervals of length $x^{1-\delta}$; that is,
\begin{equation}
\label{eqn:Hoh}
\sum_{x<p\leq x+h}\log p\sim h,\qquad x^{1-\delta}\leq h\leq x.
\end{equation}
It was quite stunning when Hoheisel \cite{Hoheisel} proved that \eqref{eqn:Hoh} holds {\it unconditionally} for any $\delta \leq 1/33000$; this has been improved to $\delta\leq \frac{5}{12}$ \cite{MR953665}.  Hoheisel proved \eqref{eqn:Hoh} using the bound
\begin{equation}
\label{eqn:ZDE_Zeta}
N(\sigma,T):=\#\{\rho=\beta+i\gamma\colon \beta\geq\sigma,~|\gamma|\leq T,~\zeta(\rho)=0\}\ll T^{4\sigma(1-\sigma)}(\log T)^{13}
\end{equation}
(a zero density estimate for $\zeta(s)$) and an explicit version of Littlewood's zero-free region
\begin{equation}
\label{eqn:Littlewood}
\zeta(s)\neq 0,\qquad \re(s)\geq 1-\frac{c\log\log(|\im(s)|+e)}{\log(|\im(s)|+e)},
\end{equation}
where $c>0$ is an absolute and effectively computable constant.  This is an improvement over the ``standard'' zero-free region
\[
\zeta(s)\neq 0,\qquad \re(s)\geq 1-\frac{c}{\log(|\im(s)|+e)}
\]
proved by de la Vall{\'e}e Poussin.

Here, we study a broad generalization of Hoheisel's work.  Suppose that an object $\pi$ (e.g., a number field, abelian variety, automorphic form) gives rise to a Dirichlet series
\[
L(s,\pi) = \sum_{n=1}^{\infty} \frac{\lambda_{\pi}(n)}{n^s}
\]
satisfying the {\bf Hoheisel property}; that is, the following conditions hold:
\begin{enumerate}
	\item $\lambda_{\pi}(n)\geq 0$ for all $n$.
	\item We have an ``explicit formula''
	\[
\sum_{p\leq x}\lambda_{\pi}(p)\log p = x - \sum_{\substack{\beta\geq 0 \\ |\gamma|\leq T}}\frac{x^{\beta+i\gamma}}{\beta+i\gamma}+O_{\pi}\Big(\frac{x(\log T x)^2}{T}\Big),
\]
where $1\leq T\leq x$ and $\beta+i\gamma$ denotes a zero of $L(s,\pi)$.
\item We have $L(1+it,\pi)\neq 0$ for all $t\in\R$.  Also, there exists a constant $a_{\pi}>0$, depending at most on $\pi$, such that $L(s,\pi)$ is nonzero in the region
\begin{equation}
\label{eqn:Standard_ZFR}
\re(s)\geq 1-\frac{a_{\pi}}{\log(|\im(s)|+e)}
\end{equation}
apart from $O_{\pi}(1)$ exceptional zeroes $\beta+i\gamma$ that satisfy $|\gamma|\ll_{\pi} 1$.
\item If $N_{\pi}(\sigma,T)$ is the number of zeroes $\rho=\beta+i\gamma$ of $L(s,\pi)$ such that $\beta\geq\sigma$ and $|\gamma|\leq T$, then there exist constants $c_{\pi}>0$, depending at most on $\pi$, such that
\begin{equation}
\label{eqn:lfZde}
	N_{\pi}(\sigma,T)\ll_{\pi} T^{c_{\pi}(1-\sigma)}.
\end{equation}
\item We have $N(0,T)\ll_{\pi} T\log T$.
\end{enumerate}
Moreno \cite{Moreno_Hoheisel} proved that if $L(s,\pi)$ satisfies the Hoheisel property, then there exists a constant $0<\delta_{\pi}<1$  (depending at most on $\pi$) such that
\begin{equation}
\label{eqn:moreno_lower_bound}
\sum_{x<p\leq x+h}\lambda_{\pi}(p)\log p\gg_{\pi} h,\qquad h\geq x^{1-\delta_{\pi}}.
\end{equation}
Moreno referred to this as the {\bf Hoheisel phenomenon}.

A key point here is that while the ``standard'' zero-free region \eqref{eqn:Standard_ZFR} is inferior to Littlewood's region \eqref{eqn:Littlewood} in the dependence on $|\im(s)|$, one can still prove a lower bound on \eqref{eqn:moreno_lower_bound} of the expected order when the zero density estimate \eqref{eqn:lfZde} is {\it log-free}, in the sense that there are no factors of $\log T$ (in contrast with \eqref{eqn:ZDE_Zeta}).  The absence of the logarithmic factors in \eqref{eqn:lfZde} serves as a proxy for a zero-free region as strong as \eqref{eqn:Littlewood}.  However, even with a log-free zero density estimate at one's disposal, Moreno's proof suggests that if $h$ is as small as $x^{1-\delta}$, then one must be able to take $a_{\pi}$ arbitrarily large in \eqref{eqn:Standard_ZFR} in order to replace the lower bound \eqref{eqn:moreno_lower_bound} with an asymptotic.  Such a zero-free region appears to be well beyond the reach of current methods.

Akbary and Trudgian \cite{AT} proved that certain $L$-functions arising from automorphic representations achieve the Hoheisel phenomenon.  To describe their results, let $\mathbb{A}_{F}$ be the ring of ad\`{e}les over a number field $F$, and for an integer $n \geq 1$, let $\mathfrak{F}_n$ be the set of cuspidal automorphic representations $\pi$ of $\GL_n(\mathbb{A}_{F})$ with arithmetic conductor $\kq_{\pi}$ and unitary central character.  We implicitly normalize the central character to be trivial on the product of positive reals when embedded diagonally into the archimedean places of the id\`{e}les $\mathbb{A}_F^{\times}$, so that $\mathfrak{F}_n$ is discrete.  Let $\tilde{\pi}$ be the representation contragredient to $\pi$.  To each $\pi\in\mathfrak{F}_n$, there is an associated standard $L$-function
\[
L(s,\pi)=\sum_{\kn}\frac{\lambda_{\pi}(\kn)}{\N \kn^s}=\prod_{\kp}\prod_{j=1}^n (1-\alpha_{j,\pi}(\kp)\N\kp^{-s})^{-1},\qquad \re(s)>1.
\]
Here, $\kn$ (resp.~$\kp$) runs through the nonzero integral (resp.~prime) ideals of $\cO_F$, the ring of integers of $F$.  The $L$-function $L(s,\pi)$ has an analytic continuation and functional equation similar to that of $\zeta(s)$.

Consider the Rankin--Selberg $L$-function
\[
L(s,\pi\times\tilde{\pi})=\sum_{\kn}\frac{\lambda_{\pi\times\tilde{\pi}}(\kn)}{\N\kn^s}=\prod_{\kp}\prod_{j=1}^n\prod_{j'=1}^n (1-\alpha_{j,j',\pi\times\tilde{\pi}}(\kp)\N\kp^{-s})^{-1},\qquad \re(s)>1,
\]
which also has an analytic continuation and functional equation.  If $\kp\nmid\kq_{\pi}$, then we have
\[
\{\alpha_{j,j',\pi\times\tilde{\pi}}(\kp)\colon 1\leq j,j'\leq n\}=\{\alpha_{j,\pi}(\kp)\overline{\alpha_{j',\pi}(\kp)}\colon 1\leq j,j'\leq n\}.
\]
We define the numbers $\Lambda_{\pi\times\widetilde{\pi}}(\mathfrak{n})$ by the Dirichlet series identity
	\begin{equation}
	\label{eqn:L'/L}
	-\frac{L'}{L}(s,\pi\times\tilde{\pi}) = \sum_{\mathfrak{n}}\frac{\Lambda_{\pi\times\widetilde{\pi}}(\mathfrak{n})}{\mathrm{N}\mathfrak{n}^s}=\sum_{\mathfrak{p}}\sum_{k=1}^{\infty}\frac{\sum_{1\leq j,j'\leq n}\alpha_{j,j',\pi\times\tilde{\pi}}(\mathfrak{p})^k\log\mathrm{N}\mathfrak{p}}{\mathrm{N}\mathfrak{p}^{ks}},
	\end{equation}
	where convergence is absolute for $\mathrm{Re}(s)>1$.  The numbers $\Lambda_{\pi\times\widetilde{\pi}}(\mathfrak{n})$ are nonnegative [10, Lemma a], and they equal zero when $\kn$ is not a power of a prime ideal.  Note that $\Lambda_{\pi\times\tilde{\pi}}(\kp)=\lambda_{\pi\times\tilde{\pi}}(\kp)\log\N\kp$.  Also, if $\kp\nmid\kq_{\pi}$, then $\lambda_{\pi\times\tilde{\pi}}(\kp)=|\lambda_{\pi}(\kp)|^2$.  The generalized Ramanujan conjecture (which we abbreviate to GRC) predicts that $|\alpha_{j,\pi}(\kp)|=1$ whenever $\kp\nmid\kq_{\pi}$ and $|\alpha_{j,\pi}(\kp)|\leq1$ otherwise.

Until now, a ``standard'' zero-free region of the shape \eqref{eqn:Standard_ZFR} was  known for $L(s,\pi\times\tilde{\pi})$ only when $\pi$ is self-dual, so that $\pi=\tilde{\pi}$ (see Brumley \cite[Theorem A.1]{Humphries}), and a hypothesis slightly weaker than GRC suffices to prove such a zero-free region when $\pi\neq\tilde{\pi}$ \cite{Humphries}.  A zero-free region of the shape \eqref{eqn:Littlewood} seems to be currently out of reach when $n>1$ unless $\pi$ is induced, via automorphic induction, by a one-dimensional representation over a cyclic Galois extension of $F$.

When $F=\Q$, Akbary and Trudgian \cite{AT} proved that if $L(s,\pi\times\tilde{\pi})$ has a zero-free region of the shape \eqref{eqn:Standard_ZFR} and there exists a constant $0<\alpha_{\pi}<\frac{1}{2}$ such that the upper bound
\begin{equation}
\label{eqn:AT_hypothesis}
\sum_{x<\N\kn\leq x+h}\Lambda_{\pi\times\tilde{\pi}}(\kn)\ll_{\pi} h,\qquad x^{1-\alpha_{\pi}}\leq h\leq x
\end{equation}
holds, then one can prove a log-free zero density estimate for $L(s,\pi\times\tilde{\pi})$ like \eqref{eqn:lfZde}.  This leads to a result of the following shape: there exists a constant $0<\delta_{\pi}<1$ such that
\begin{equation}
\label{eqn:RJLOT}
\sum_{x<\N\kn\leq x+h}\Lambda_{\pi\times\tilde{\pi}}(\kn)\asymp_{\pi} h,\qquad x^{1-\delta_{\pi}}\leq h\leq x.
\end{equation}
One may think of the hypothesis \eqref{eqn:AT_hypothesis} as an average form of GRC.  If we assume GRC in full, then the contribution from the terms for which $\kn = \kp^k$ with $k\geq 2$ is negligible, and
\begin{equation}
\label{eqn:auto_hoh_phen}
\sum_{x<\N\kp\leq x+h}|\lambda_{\pi}(\kp)|^2\log\N\kp\asymp_{\pi} h,\qquad x^{1-\delta_{\pi}}\leq h\leq x.
\end{equation}

Around the same time as Akbary and Trudgian's work, Motohashi \cite{Motohashi} unconditionally proved a refined version of \eqref{eqn:auto_hoh_phen} when $F=\Q$, $n=2$, and $\pi$ corresponds to a level 1 Hecke--Maa\ss{} cusp form.  Shortly afterward, Lemke Oliver and Thorner \cite{RJLOT} proved that \eqref{eqn:RJLOT} holds for $n\geq 1$ and all $F$ without appealing to \eqref{eqn:AT_hypothesis}, regardless of whether $\pi\in\mathfrak{F}_n$ is self-dual, provided that there exists a (noncuspidal) automorphic representation $\pi\boxtimes\widetilde{\pi}$ of $\mathrm{GL}_{m^2}(\mathbb{A}_F)$ such that $L(s,\pi\boxtimes\widetilde{\pi})=L(s,\pi\times\widetilde{\pi})$.  This is predicted by Langlands functoriality but is only known in special cases.  For instance, this is not even known for an arbitrary $\pi\in\mathfrak{F}_3$.

\section{Main results}

In this paper, we prove an {\it unconditional} proof of \eqref{eqn:RJLOT} in a more precise form.  Our result also exhibits effective dependence on the analytic conductor $C(\pi)$ of $\pi$ (see \eqref{eqn:analytic_conductor_def}) in the spirit of Linnik's bound on the least prime in an arithmetic progression \cite{Linnik}.
\begin{theorem}
	\label{thm:main_theorem}
	Let $\pi\in\mathfrak{F}_n$.  There exist positive, absolute, and effectively computable constants $\Cl[abcon]{Hoheisel1}$, $\Cl[abcon]{Hoheisel2}$, $\Cl[abcon]{Hoheisel3}$, $\Cl[abcon]{Hoheisel4}$, $\Cl[abcon]{Hoheisel5}$, and $\Cl[abcon]{Hoheisel6}$ such that the following are true.
	\begin{enumerate}
	\item The Rankin--Selberg $L$-function $L(s,\pi\times\tilde{\pi})$ has at most one zero in the region
	\[
	\re(s)\geq 1-\frac{\Cr{Hoheisel1}}{\log(C(\pi)^n (|\im(s)|+e)^{n^2[F:\Q]})}.
	\]
	If such an exceptional zero $\beta_1$ exists, then it must be real and simple and satisfy $\beta_1\leq 1-C(\pi)^{-\Cr{Hoheisel2}n}$.
	
	\item Let $A\geq \Cr{Hoheisel3}$, $\log\log C(\pi)\geq \Cr{Hoheisel4}n^4[F:\Q]^2$, and $x\geq C(\pi)^{\Cr{Hoheisel5}A^2n^3[F:\Q]\log(e n[F:\Q])}$.  If
	\[
\delta=\frac{1}{16An^2[F:\Q]\log(e n[F:\Q])}
	\]
	and $x^{1-\delta}\leq h\leq x$, then
	\begin{equation}
	\label{eqn:auto_hoheisel_1}
	\sum_{x<\N\kn\leq x+h}\Lambda_{\pi\times\tilde{\pi}}(\kn)=\begin{cases}
		h(1-\xi^{\beta_1-1})(1+O(e^{-\Cr{Hoheisel6}A}))&\mbox{if $\beta_1$ exists,}\\
		h(1+O(e^{-\Cr{Hoheisel6}A}))&\mbox{otherwise.}
	\end{cases}
	\end{equation}
	The implied constant is absolute, and $\xi\in[x,x+h]$ satisfies $(x+h)^{\beta_1}-x^{\beta_1}=\beta_1 h \xi^{\beta_1-1}$.
	\end{enumerate}
\end{theorem}
\begin{remark}
Our assumed lower bound on $C(\pi)$ simplifies several aspects of the proof.  It can be removed with additional effort, but the dependence on $n$ and $[F:\Q]$ will change.
\end{remark}

Note that since $\Lambda_{\pi\times\tilde{\pi}}(\kn)\geq 0$ for all $\kn$ and $\Lambda_{\pi\times\tilde{\pi}}(\kp)=|\lambda_{\pi}(\kp)|^2\log\N\kp$ for $\kp\nmid\kq_{\pi}$, \cref{thm:main_theorem} (along with the Luo--Rudnick--Sarnak bound \eqref{eqn:LRS_2} to handle the $\kp \mid \kq_{\pi}$) implies the bound
\[
\sum_{\substack{x<\N\kp\leq x+h}}|\lambda_{\pi}(\kp)|^2\log\N\kp\leq \begin{cases}
		h(1-\xi^{\beta_1-1})(1+O(e^{-\Cr{Hoheisel6}A}))&\mbox{if $\beta_1$ exists,}\\
		h(1+O(e^{-\Cr{Hoheisel6}A}))&\mbox{otherwise}
	\end{cases}
\]
under the hypotheses of \cref{thm:main_theorem}.  We do not have the corresponding lower bound because we cannot rule out the possibility that the contribution from prime powers is $\gg h$ due to insufficient progress toward GRC. In contexts where a prime power contribution is expected to be small but GRC is not yet known, it often suffices to establish the ``Hypothesis H'' of Rudnick and Sarnak \cite{RS}, which asserts that for any fixed $k\geq 2$, we have
\[
\sum_{\kp}|\Lambda_{\pi}(\kp^k)|^2\N\kp^{-k}<\infty.
\]
(The original hypothesis is stated over $\Q$, but the extension to number fields incurs no complications.)  Hypothesis H is known when $\pi\in\mathfrak{F}_n$ and $1\leq n\leq 4$, along with a few other special cases \cite{Kim,RS,MR2364718}.  While Hypothesis H on its own is not enough to ensure that the contributions from higher prime powers in \eqref{eqn:auto_hoheisel_1} are negligible, the progress toward Langlands functoriality that leads to proofs of Hypothesis H when $n\leq 4$ also leads to the following strong form of the Hoheisel phenomenon for $\pi\in\mathfrak{F}_n$ with $n\in\{1,2,3,4\}$.

\begin{theorem}
\label{hypH}
	Let $n\in\{1,2,3,4\}$ and $\pi\in\mathfrak{F}_n$.  With the notation and hypotheses of \cref{thm:main_theorem}, we have
	\[
	\sum_{x<\N\kn\leq x+h}|\lambda_{\pi}(\kp)|^2\log\N\kp=\begin{cases}
		h(1-\xi^{\beta_1-1})(1+O(e^{-\Cr{Hoheisel6}A}))&\mbox{if $\beta_1$ exists,}\\
		h(1+O(e^{-\Cr{Hoheisel6}A}))&\mbox{otherwise.}
	\end{cases}
	\]
\end{theorem}
\begin{remark}
Taking $n=2$ and $F=\Q$, we recover Motohashi's result in \cite{Motohashi}.	Also, if $n\geq 1$ and $\pi\in\mathfrak{F}_n$ satisfies the averaged form of GRC in \eqref{eqn:secondorder} below, then \cref{hypH} will hold for $\pi$.  However, the implied constant will depend on $n$ in accordance with the implied constant in \eqref{eqn:secondorder}.
\end{remark}

As in \cite{AT,RJLOT,Moreno_Hoheisel,Motohashi}, one must have a standard zero-free region and a log-free zero density estimate for $L(s,\pi\times\tilde{\pi})$.  Our log-free zero density estimate is proved using the ideas in Soundararajan and Thorner \cite{ST}.  However, an unconditional standard zero-free region for $L(s,\pi\times\tilde{\pi})$ has not yet appeared in the literature.  This is due to the fact that most previous proofs of a standard zero-free region for an $L$-function $L(s,\Pi)$, including \cite[Proof of Theorem 5.10]{IK} (which is based on \cite[Appendix]{HL}), have required that the Rankin--Selberg $L$-functions $L(s,\Pi \times \Pi)$, $L(s,\tilde{\Pi} \times \tilde{\Pi})$, and $L(s,\Pi \times \tilde{\Pi})$ exist, and this is not yet known to be the case when $\Pi$ is the Rankin--Selberg convolution $\pi \times \tilde{\pi}$ except when $n \in \{1,2\}$. The only exception is the aforementioned work of Brumley when $\pi$ is self-dual; this avoids these requirements but is contingent on the assumption that $\pi$ is self-dual. Humphries \cite{Humphries} recently proved the existence of a constant $c_{\pi}>0$ (depending at most on $\pi$) such that if $|\alpha_{j,\pi}(\kp)|\leq 1$ for almost all $\kp$, then $L(s,\pi\times\tilde{\pi})$ has a zero-free region of the shape
\[
\re(s)\geq 1-\frac{c_{\pi}}{\log(|\im(s)|+e)},\qquad \im(s)\neq 0.
\]
Our work addresses the remaining cases using the fact that $\pi\times\tilde{\pi}$ is self-dual even if $\pi$ is not.

As part of our proofs, we supply an unconditional standard zero-free region for $L(s,\pi\times\tilde{\pi})$ (apart from a possible Landau--Siegel zero) with good uniformity in the analytic conductor.  In order to ensure that our results are completely effective (even if a Landau--Siegel zero exists), we also prove a uniform version of Deuring and Heilbronn's observation that Landau--Siegel zeroes tend to repel other zeroes away from the line $\re(s)=1$.

\section{Properties of \texorpdfstring{$L$}{L}-functions}
\label{sec:L-functions}


We recall some standard facts about $L$-functions arising from automorphic representations and their Rankin--Selberg convolutions; see \cite{Brumley,GJ2,JPSS,MW,ST}.

\subsection{Standard \texorpdfstring{$L$}{L}-functions}

Let $\pi=\bigotimes_{\kp}\pi_{\kp} \in \mathfrak{F}_n$ be a cuspidal automorphic representation of $\mathrm{GL}_n(\mathbb{A}_F)$.  Let $\kq_{\pi}$ be the conductor of $\pi$.  The local standard $L$-function $L(s,\pi_{\kp})$ at a prime ideal $\kp$ is defined in terms of the Satake parameters $\{\alpha_{1,\pi}(\kp),\ldots,\alpha_{n,\pi}(\kp)\}$ by
\begin{equation}
	\label{eqn:Euler_p_single}
	L(s,\pi_{\kp})=\prod_{j=1}^{n}\Big(1-\frac{\alpha_{j,\pi}(\kp)}{\N\kp^{s}}\Big)^{-1}=\sum_{k=0}^{\infty}\frac{\lambda_{\pi}(\kp^k)}{\N\kp^{ks}}.
\end{equation}
We have $\alpha_{j,\pi}(\kp)\neq0$ for all $j$ whenever $\kp\nmid\kq_{\pi}$, whereas it may be the case that $\alpha_{j,\pi}(\kp)=0$ for at least one $j$ when $\kp \mid \kq_{\pi}$.  The standard $L$-function $L(s,\pi)$ associated to $\pi$ is of the form
\[
L(s,\pi)=\prod_{\kp} L(s,\pi_{\kp})=\sum_{\kn}\frac{\lambda_{\pi}(\kn)}{\N\kn^s}.
\]
The Euler product and Dirichlet series converge absolutely when $\re(s)>1$.

At each archimedean place $ v$ of $F$, there are $n$ Langlands parameters $\mu_{j,\pi}(v)\in\mathbb{C}$, from which we define
\[
L(s,\pi_{\infty}) = \prod_{ v}\prod_{j=1}^{n}\Gamma_{ v}(s+\mu_{j,\pi}( v)),\qquad \Gamma_{ v}(s):=\begin{cases}
	\pi^{-s/2}\Gamma(s/2)&\mbox{if $F_{ v}=\R$,}\\
	2(2\pi)^{-s}\Gamma(s)&\mbox{if $F_{ v}=\mathbb{C}$.}
\end{cases}
\]
Luo, Rudnick, and Sarnak \cite{LRS} and Mueller and Speh \cite{MS} proved the uniform bounds
\begin{equation}
\label{eqn:LRS_finite}
	|\alpha_{j,\pi}(\kp)|\leq  \N\kp^{\theta_n}\quad\textup{and}\quad\re(\mu_{j,\pi}( v))\geq -\theta_n,\qquad \theta_n = \frac{1}{2} - \frac{1}{n^2 + 1}
\end{equation}
The generalized Selberg eigenvalue conjecture and GRC assert that we have $\theta_n=0$ in \eqref{eqn:LRS_finite}.

Let $\tilde{\pi}\in\mathfrak{F}_n$ be the cuspidal automorphic representation contragredient to $\pi$. We have $\kq_{\pi}=\kq_{\tilde{\pi}}$, and for each $\kp\nmid \kq_{\pi}$, we have the equalities of sets $\{\alpha_{j,\tilde{\pi}}(\kp)\}=\{\overline{\alpha_{j,\pi}(\kp)}\}$.

The completed standard $L$-function
\[
\Lambda(s,\pi) = (D_F^n \N \kq_{\pi})^{s/2} L(s,\pi)L(s,\pi_{\infty})
\]
is entire of order $1$, and there exists a complex number $\varepsilon(\pi)$ of modulus $1$ such that for all $s\in\mathbb{C}$, we have the functional equation $\Lambda(s,\pi)=\varepsilon(\pi)\Lambda(1-s,\tilde{\pi})$.

Let $d( v)=1$ if $F_{ v}=\R$ and $d( v)=2$ if $F_{ v}=\mathbb{C}$.  We define the analytic conductor of $\pi$ to be
\begin{equation}
\label{eqn:analytic_conductor_def}
C(\pi,t):=D_F^n \N\kq_{\pi}\prod_{ v}\prod_{j=1}^n(e+|it+\mu_{j,\pi}( v)|^{d( v)}),\qquad C(\pi):=C(\pi,0).
\end{equation}

\subsection{Rankin--Selberg \texorpdfstring{$L$}{L}-functions}
\label{subsec:RS}

Let $\pi\in\mathfrak{F}_n$.  The local Rankin--Selberg $L$-function $L(s,\pi_{\kp}\times\tilde{\pi}_{\kp})$ is defined at a prime ideal $\kp$ by
\begin{equation}
\label{eqn:RS_Dirichlet_series}
L(s,\pi_{\kp}\times\tilde{\pi}_{\kp})=\prod_{j=1}^{n}\prod_{j'=1}^{n}(1-\alpha_{j,j',\pi\times\tilde{\pi}}(\kp) \N\kp^{-s})^{-1}=\sum_{k=0}^{\infty}\frac{\lambda_{\pi\times\tilde{\pi}}(\kp^k)}{\N\kp^{ks}}.
\end{equation}
for suitable Satake parameters $\alpha_{j,j',\pi\times\tilde{\pi}}(\kp)$.  If $\kp\nmid \kq_{\pi}$, then we have the equality of sets
\begin{equation}
\label{eqn:separate_dirichlet_coeffs}
\{\alpha_{j,j',\pi\times\tilde{\pi}}(\kp)\}=\{\alpha_{j,\pi}(\kp)\alpha_{j',\tilde{\pi}}(\kp)\} = \{\alpha_{j,\pi}(\kp) \overline{\alpha_{j',\pi}(\kp)}\}.
\end{equation}
See \cite[Appendix]{ST} for a complete description of the numbers $\alpha_{j,j',\pi\times\tilde{\pi}}(\kp)$ even when $\kp \mid \kq_{\pi}$.  The Rankin-Selberg $L$-function $L(s,\pi\times\tilde{\pi})$ associated to $\pi$ and $\tilde{\pi}$ is of the form
\[
L(s,\pi\times\tilde{\pi})=\prod_{\kp}L(s,\pi_{\kp}\times\tilde{\pi}_{\kp})=\sum_{\kn}\frac{\lambda_{\pi\times\tilde{\pi}}(\kn)}{\N\kn^s}.
\]

Bushnell and Henniart \cite{BH} proved that the conductor $\mathfrak{q}_{\pi\times\widetilde{\pi}}$ divides $\mathfrak{q}_{\pi}^{2n-1}$.  At an archimedean place $v$ of $F$, there are $n^{2}$ complex Langlands parameters $\mu_{j,j',\pi\times\tilde{\pi}}( v)$, from which we define
\[
L(s,\pi_{\infty}\times\tilde{\pi}_{\infty}) = \prod_{ v}\prod_{j=1}^{n}\prod_{j'=1}^{n}\Gamma_{ v}(s+\mu_{j,j',\pi\times\tilde{\pi}}(v)).
\]
If $\pi$ is unramified at $v$, then we have the equality of sets
\[\{\mu_{j,j',\pi\times\tilde{\pi}}( v)\} = \{\mu_{j,\pi}( v) + \mu_{j',\tilde{\pi}}(v)\}.\]
Using the explicit descriptions of $\alpha_{j,j',\pi\times\tilde{\pi}}(\kp)$ and $\mu_{j,j',\pi\times\tilde{\pi}}(v)$ in \cite{Humphries,ST}, one sees that
\begin{equation}
\label{eqn:LRS_2}
|\alpha_{j,j',\pi\times\tilde{\pi}}(\kp)|\leq\N\kp^{2\theta_n},\qquad \re(\mu_{j,j',\pi\times\tilde{\pi}}(v))\geq -2\theta_n.
\end{equation}
The completed Rankin--Selberg $L$-function
\[\Lambda(s,\pi\times\tilde{\pi}) = (D_F^{n^{2}}\N\kq_{\pi\times\tilde{\pi}})^{s/2} s(s-1)L(s,\pi\times\tilde{\pi})L(s,\pi_{\infty}\times\tilde{\pi}_{\infty})\]
is entire of order $1$, and there exists a number $\varepsilon(\pi\times\tilde{\pi}) \in \{\pm 1\}$ such that $\Lambda(s,\pi\times\tilde{\pi})$ satisfies the functional equation
\[
\Lambda(s,\pi\times\tilde{\pi}) = \varepsilon(\pi\times\tilde{\pi}) \Lambda(1-s,\pi\times\tilde{\pi}).
\]

As with $L(s,\pi)$, we define the analytic conductor
\[
C(\pi\times\tilde{\pi},t):=D_F^{n^{2}}\N\kq_{\pi\times\tilde{\pi}}\prod_{ v}\prod_{j=1}^n \prod_{j'=1}^{n}(e+|it+\mu_{j,j',\pi\times\tilde{\pi}}( v)|^{d( v)}),\qquad C(\pi\times\tilde{\pi}):=C(\pi\times\tilde{\pi},0).
\]
The work of Bushnell and Henniart \cite{BH} and the proofs in Brumley \cite[Appendix]{Humphries} yields
\begin{equation}
\label{eqn:BH}
C(\pi\times\tilde{\pi},t)\leq C(\pi\times\tilde{\pi})(e+|t|)^{[F:\Q] n^{2}},\qquad C(\pi\times\tilde{\pi})\leq e^{O(n)}C(\pi)^{2n}.
\end{equation}

%

\section{A Brun--Titchmarsh bound}

We require a Brun--Titchmarsh type bound for the coefficients $\Lambda_{\pi\times\tilde{\pi}}(\kn)$.  In \cite[Theorem 2.4]{ST}, it is shown that if $F=\Q$, $\pi\in\mathfrak{F}_m$, $x\gg_m C(\pi\times\tilde{\pi})^{36m^2}$, and $1\leq T\leq x^{\frac{1}{9m^2}}$, then
\[
\sum_{x<n\leq xe^{1/T}}\Lambda_{\pi\times\tilde{\pi}}(n)\ll_m\frac{x}{T}.
\]
In this section, we prove a field-uniform version of this bound for $\pi\in\mathfrak{F}_n$ where the dependence of the implied constant on $n$ and $[F:\Q]$ is made clear.  In \cite{ST}, the implied constant was not pertinent, but here, the dependence of the implied constant on $n$ and $[F:\Q]$ impacts exactly how large we may take $\delta$ in \cref{thm:main_theorem}.

\begin{proposition}
\label{prop:BT}
	Let $\pi\in\mathfrak{F}_n$.  Suppose that $\log\log C(\pi\times\tilde{\pi})\gg n^4[F:\Q]^2$ with a sufficiently large implied constant.  If
	\[
	x\geq e^{O(n^4[F:\Q]^2)}C(\pi\times\tilde{\pi})^{32n^2[F:\Q]}\qquad\textup{and}\qquad 1\leq T\leq x^{\frac{1}{16n^2[F:\Q]}},
	\]
	then
	\[
	\sum_{x<\N\kn\leq xe^{1/T}}\Lambda_{\pi\times\tilde{\pi}}(\kn)\ll n^2[F:\Q]\frac{x}{T}.
	\]
\end{proposition}

Much like the work in \cite[Section 6]{ST}, we use the Selberg sieve.  The primary difference is the attention we pay to the dependence of implied constants on $n$ and $[F:\Q]$.  We begin with an effective bound for $L(s,\pi\times\tilde{\pi})$ with the dependence on $n$ and $[F:\Q]$ made clear.

\begin{lemma}
	\label{lem:convexity}
	Let $\pi\in\mathfrak{F}_n$.  If $\epsilon>0$ and
	\[
	\log\log C(\pi\times\tilde{\pi})\gg\frac{n^2[F:\Q]}{\epsilon}
	\]
	with a sufficiently large implied constant, then for $\frac{1}{2}\leq\sigma\leq 1$ and $t\in\R$, we have the bound
	\[
	\lim_{\sigma'\to\sigma^+}|(1-\sigma')L(\sigma'+it,\pi\times\tilde{\pi})|\ll e^{O(n^2[F:\Q](1-\sigma))}(C(\pi\times\tilde{\pi})(1+|t|)^{n^2[F:\Q]})^{\frac{1-\sigma}{2}+\epsilon}.
	\]
\end{lemma}
\begin{proof}
	It follows from the work of Li \cite[Theorem 2]{Li} (with straightforward changes in order to apply to arbitrary number fields $F$) that there exists an absolute constant $\Cl[abcon]{Li}>0$ such that
	\[
	\lim_{\sigma'\to\sigma^+}|(1-\sigma')L(\sigma',\pi\times\tilde{\pi})|\ll \exp\Big(\Cr{Li}n^2[F:\Q]\frac{\log C(\pi\times\tilde{\pi})}{\log\log C(\pi\times\tilde{\pi})}\Big),\qquad 1\leq\sigma\leq 3.
	\]
	If $\log\log C(\pi\times\tilde{\pi})\geq 2\Cr{Li}n^2[F:\Q]/\epsilon$, then
	\[
	\exp\Big(\Cr{Li}n^2[F:\Q]\frac{\log C(\pi\times\tilde{\pi})}{\log\log C(\pi\times\tilde{\pi})}\Big)\leq C(\pi\times\tilde{\pi})^{\frac{\epsilon}{2}}.
	\]
	It follows from work of Soundararajan and Thorner \cite[Theorem 1.1 with $\delta=0$]{ST} (with straightforward changes in order to apply to arbitrary number fields $F$) and the above analysis that
	\[
	|L(\tfrac{1}{2},\pi\times\tilde{\pi})|\ll e^{O(n^2[F:\Q])}C(\pi\times\tilde{\pi})^{\frac{1}{4}}|L(\tfrac{3}{2},\pi\times\tilde{\pi})|^2\leq e^{O(n^2[F:\Q])}C(\pi\times\tilde{\pi})^{\frac{1}{4}+\epsilon}.
	\]
	By the Phragm{\'e}n--Lindel{\"o}f principle, we have the bound
	\[
	\lim_{\sigma'\to \sigma^+}|(1-\sigma')L(\sigma',\pi\times\tilde{\pi})|\leq e^{O(n^2[F:\Q](1-\sigma))}C(\pi\times\tilde\pi)^{\frac{1-\sigma}{2}+\epsilon},\qquad \tfrac{1}{2}\leq\sigma\leq 1.
	\]
	Since our results are uniform in the analytic conductor $C(\pi\times\tilde{\pi})$, and hence in the spectral parameters $\mu_{j,j',\pi\times\tilde{\pi}}(v)$, we can shift all of the spectral parameters by $it$, thus proving that
	\[
	\lim_{\sigma'\to \sigma^+}\Big|\frac{1-\sigma'-it}{1+\sigma'+it}L(\sigma'+it,\pi\times\tilde{\pi})\Big|\leq e^{O(n^2[F:\Q](1-\sigma))}C(\pi\times\tilde\pi,t)^{\frac{1-\sigma}{2}+\epsilon},\qquad \tfrac{1}{2}\leq\sigma\leq 1.
	\]
	We conclude the desired result by invoking \eqref{eqn:BH}.
\end{proof}

For a squarefree integral ideal $\kd$ of $\cO_F$, define
\[
g_{\kd}(s,\pi\times\tilde{\pi}):=\prod_{\kp|\kd}(1-L(s,\pi_{\kp}\times\tilde{\pi}_{\kp})^{-1}),\qquad g(\kd):=g_{\kd}(1,\pi\times\tilde{\pi}).
\]
We require some estimates for $g_{\kd}(s,\pi\times\tilde{\pi})$ and $g(\kd)$.
\begin{lemma}
\label{lem:g_def}
	Suppose that $\log\log C(\pi\times\tilde{\pi})\gg n^4[F:\Q]^2$ with a sufficiently large implied constant.  Let $\kd\neq \cO_F$ be a squarefree integral ideal.  We have $0\leq g(\kd)<1$, $g(\cO_F)=1$, and
	\[
	|g_{\kd}(s,\pi\times\tilde{\pi})|\leq C(\pi\times\tilde{\pi})^{\frac{1}{8n^2[F:\Q]}}\N\kd^{\frac{1}{4}},\qquad \re(s)= 1-\frac{1}{2n^2[F:\Q]}.
	\]
\end{lemma}
\begin{proof}
	The bound on $g(\kd)$ follows immediately from \eqref{eqn:LRS_2}.  Also, by \eqref{eqn:LRS_2} and the fact that $\N\kp\geq 2$ for all prime ideals $\kp$, we have the bound
	\begin{align*}
	|g_{\kd}(1-\tfrac{1}{2n^2[F:\Q]}+it,\pi\times\tilde{\pi})|\leq \prod_{\kp|\kd}\Big(1+\Big(1+\frac{\N\kp^{1-\frac{2}{n^2+1}}}{\N\kp^{1-\frac{1}{2n^2[F:\Q]}}}\Big)^{n^2}\Big)\leq \prod_{\kp|\kd}2^{n^2+2}.
	\end{align*}
	The proof of \cite[Lemma 1.13b]{Weiss} shows that for all $\epsilon>0$, the number of distinct prime ideal divisors of $\kd$ is bounded by $6e^{2/\epsilon}[F:\Q]+\epsilon\log\N\kd$.  We apply this with $\epsilon=\frac{1}{4(n^2+2)}$ to bound the above display by
	\[
	\N\kd^{\frac{1}{4}} e^{6e^{8(n^2+2)}(n^2+2)[F:\Q]}.
	\]
	This is bounded as claimed when the implied constant for the lower bound on $C(\pi\times\tilde{\pi})$ is made sufficiently large.
\end{proof}

Let $\Phi$ be a smooth nonnegative function supported in $(-2,2)$, and let
\begin{equation}
\label{eqn:check_def}
{\check \Phi}(s) = \int_{-\infty}^{\infty}\Phi(y)e^{sy}dy.
\end{equation}
Then ${\check \Phi}(s)$ is entire, and for any integer $k\geq 1$, integration by parts yields
\begin{equation}
\label{eqn:check_bound}
|{\check \Phi}(s)|\ll_{\Phi,k}e^{2|\re(s)|}|s|^{-k}.
\end{equation}
Let $T\geq 1$.  By Mellin inversion, we have
\[
\Phi(T\log x)=\frac{1}{2\pi i T}\int_{c-i\infty}^{c+i\infty}{\check\Phi}(s/T)x^{-s}ds
\]
for any $x>0$ and $c\in\R$.

\begin{lemma}
\label{lem:local_density}
	Let $\pi\in\mathfrak{F}_n$, and let $\kd$ be a squarefree integral ideal of $\cO_F$.  Let $x,T\geq 1$, and let $\log\log C(\pi\times\tilde{\pi})\gg n^4[F:\Q]^2$ with a sufficiently large implied constant.  We have
	\[
	\Big|\sum_{\substack{\kd|\kn}}\lambda_{\pi\times\tilde{\pi}}(\kn)\Phi\Big(T\log\frac{\N\kn}{x}\Big)-\kappa g(\kd)\frac{x}{T}{\check\Phi}(1/T)\Big|\ll x^{1-\frac{1}{2n^2[F:\Q]}}T^{\frac{3}{8}}C(\pi\times\tilde{\pi})^{\frac{1}{2n^2[F:\Q]}}\N\kd^{\frac{1}{4}},
	\]
	where $\kappa>0$ is the residue at $s=1$ of $L(s,\pi\times\tilde{\pi})$.
\end{lemma}
\begin{proof}
The quantity to be estimated equals
\[
\frac{1}{2\pi iT}\int_{1-\frac{1}{2n^2[F:\Q]}-i\infty}^{1-\frac{1}{2n^2[F:\Q]}+i\infty}L(s,\pi\times\tilde{\pi}){\check\Phi}(s/T)x^s g_{\kd}(s,\pi\times\tilde{\pi})ds.
\]
By \cref{lem:convexity} with $\epsilon=\frac{1}{8n^2[F:\Q]}$, \cref{lem:g_def}, and \eqref{eqn:check_bound} with $k=0$ and $2$, this is
\[
\ll\frac{x^{1-\frac{1}{2n^2[F:\Q]}}}{T}C(\pi\times\tilde{\pi})^{\frac{1}{2n^2[F:\Q]}}\N\kd^{\frac{1}{4}}\int_{-\infty}^{\infty}(1+|t|)^{\frac{3}{8}}\min\Big\{1,\frac{T^2}{(1+|t|)^2}\Big\} dt,
\]
which is bounded as claimed.
\end{proof}

\begin{lemma}
\label{lem:pre-BT}
	Suppose that $\log\log C(\pi\times\tilde{\pi})\gg n^4[F:\Q]^2$ and $z\geq e^{O(n^2[F:\Q])}C(\pi\times\tilde{\pi})^{2}$, each with a sufficiently large implied constant.  If $x>0$ and $T\geq 1$, then
	\[
\sum_{\substack{\kn \\ \kp|\kn\implies \N\kp>z}}\lambda_{\pi\times\tilde{\pi}}(\kn)\Phi\Big(T\log\frac{\N\kn}{x}\Big)\leq \frac{3x}{T\log z}{\check\Phi}(1/T)+O(x^{1-\frac{1}{2n^2[F:\Q]}}T^{\frac{3}{8}}C(\pi\times\tilde{\pi})^{\frac{1}{2n^2[F:\Q]}}z^5).
\]
\end{lemma}

\begin{proof}
By proceeding as in the formulation of the Selberg sieve in \cite[Theorem 7.1]{FI} (see also \cite[Lemma 3.6]{Weiss} for a treatment with field uniformity), we find using \cref{lem:local_density} that
\begin{align*}
\sum_{\substack{\kn \\ \kp|\kn\implies \N\kp>z}}\lambda_{\pi\times\tilde{\pi}}(\kn)\Phi\Big(T\log\frac{\N\kn}{x}\Big)&\leq \kappa \frac{x}{T}{\check\Phi}(1/T)\Big(\sum_{\substack{\kd|\prod_{\N\kp\leq z}\kp \\ \N\kd<z}}\prod_{\kp|\kd}\frac{g(\kp)}{1-g(\kp)}\Big)^{-1}\\
&+O\Big(x^{1-\frac{1}{2n^2[F:\Q]}}T^{\frac{3}{8}}C(\pi\times\tilde{\pi})^{\frac{1}{2n^2[F:\Q]}}\sum_{\N\kd_1,\N\kd_2\leq z}\N(\mathrm{lcm}(\kd_1,\kd_2))^{\frac{1}{4}}\Big).
\end{align*}
We use the bound
\[
\sum_{\N\kn\leq z}1\leq (2/\epsilon)^{[F:\Q]}z^{1+\epsilon},\qquad z>0,\qquad 0<\epsilon<1
\]
in \cite[Lemma 1.12a]{Weiss} to bound the sum over $\kd_1$ and $\kd_2$ by $O(e^{O([F:\Q])}z^{5})$.  By the definitions of $g(\kd)$ and $L(s,\pi_{\kp}\times\tilde{\pi}_{\kp})$, we have the lower bound
\begin{multline*}
	\sum_{\substack{\kd|\prod_{\N\kp\leq z}\kp \\ \N\kd<z}}\prod_{\kp|\kd}\frac{g(\kp)}{1-g(\kp)}\geq \sum_{\substack{\N\kn\leq z \\ \textup{$\kn$ squarefree}}}\prod_{\kp|\kn}\sum_{j=1}^{\infty}\frac{\lambda_{\pi\times\tilde{\pi}}(\kp^j)}{\N\kp^j}\geq \sum_{\N\kn\leq z}\frac{\lambda_{\pi\times\tilde{\pi}}(\kn)}{\N\kn}\geq 1+\sum_{\sqrt{z}<\N\kn\leq z}\frac{\lambda_{\pi\times\tilde{\pi}}(\kn)}{\N\kn}.
\end{multline*}

Let $0<\epsilon_0<\frac{1}{10}$, and let $\Phi_1$ be a fixed nonnegative smooth function supported on $[0,1]$, with $\Phi_1(t)=1$ for $\epsilon_0\leq t\leq 1-\epsilon_0$ and $\Phi_1(t)\leq 1$ for $0\leq t\leq 1$.  By taking $\kd=\cO_F$ and $T=1$ in \cref{lem:local_density}, we find that if $y\geq 1$, then
\begin{multline*}
\sum_{y\leq \N\kn\leq ey}\frac{\lambda_{\pi\times\tilde{\pi}}(\kn)}{\N\kn}\geq\frac{1}{ey}\sum_{\kn}\lambda_{\pi\times\tilde{\pi}}(\kn)\Phi_1\Big(\log\frac{\N\kn}{y}\Big)\\
=\frac{e-1+O(\epsilon_0)}{e}\kappa+O(y^{1-\frac{1}{2n^2[F:\Q]}}C(\pi\times\tilde{\pi})^{\frac{1}{2n^2[F:\Q]}}).
\end{multline*}
We dyadically divide $[\sqrt{z},z]$ into subintervals of the form $[y,ey]$ and conclude that
\[
1+\sum_{\sqrt{z}<\N\kn\leq z}\frac{\lambda_{\pi\times\tilde{\pi}}(\kn)}{\N\kn}\geq 1+\frac{\kappa}{3}\log z+O(z^{-\frac{1}{4n^2[F:\Q]}}C(\pi\times\tilde{\pi})^{\frac{1}{2n^2[F:\Q]}})
\]
once $\epsilon_0$ is suitably small.  If $z\geq e^{O(n^2[F:\Q])}C(\pi\times\tilde{\pi})^{2}$, then since $\kappa>0$, it follows that
\[
\kappa\frac{x}{T}{\check\Phi}(1/T)\Big(\sum_{\substack{\kd|\prod_{\N\kp\leq z}\kp \\ \N\kd<z}}\prod_{\kp|\kd}\frac{g(\kp)}{1-g(\kp)}\Big)^{-1}\leq \frac{x}{T}{\check\Phi}(1/T)\frac{3\kappa}{1+\kappa\log z}\leq \frac{3x}{T\log z}{\check\Phi}(1/T).
\]
The result follows once we account for our lower bound for $C(\pi\times\tilde{\pi})$.
\end{proof}

\begin{proof}[Proof of \cref{prop:BT}]

We first fix $\Phi$, requiring that $0\leq\Phi(y)\leq 1$ for all $y$ and that $\Phi(y)=1$ for $y\in[0,1]$.  By \eqref{eqn:check_def}, we see that $|\Phi(1/T)|\ll_{\Phi}1$.  We choose
\[
x\geq e^{O(n^4[F:\Q]^2)}C(\pi\times\tilde{\pi})^{32n^2[F:\Q]},\qquad 1\leq T\leq x^{\frac{1}{16n^2[F:\Q]}}=z.
\]
The sum in \cref{lem:pre-BT} includes all prime powers $\kp^k$ with $\N\kp^k\in(x,xe^{1/T}]$ with $\N\kp>x^{1/(16n^2[F:\Q])}$, hence
\[
\sum_{\substack{x<\N\kp^k\leq xe^{1/T} \\ k\leq 16n^2[F:\Q]}}\lambda_{\pi\times\tilde{\pi}}(\kp^k)\ll n^2[F:\Q]\frac{x}{T\log x}.
\]
For a given $\kp$, we compare coefficients in the formal identity
\[
\exp\Big(\sum_{k=1}^{\infty}\frac{\Lambda_{\pi\times\tilde{\pi}}(\kp^k)}{k\log\N\kp}X^k\Big)=1+\sum_{k=1}^{\infty}\lambda_{\pi\times\tilde{\pi}}(\kp^k)X^k
\]
and use the nonnegativity of $\Lambda_{\pi\times\tilde{\pi}}(\kp^k)$ and $\lambda_{\pi\times\tilde{\pi}}(\kp^k)$ to deduce the bound
\[
\lambda_{\pi\times\tilde{\pi}}(\kp^k)\geq\frac{\Lambda_{\pi\times\tilde{\pi}}(\kp^k)}{k\log\N\kp},
\]
hence
\[
\sum_{\substack{x<\N\kp^k\leq xe^{1/T} \\ k\leq 16n^2[F:\Q]}}\Lambda_{\pi\times\tilde{\pi}}(\kp^k)\ll n^2[F:\Q]\frac{x}{T}.
\]

To handle the contribution when $k>16n^2[F:\Q]$, we appeal to the bound $|\Lambda_{\pi\times\tilde{\pi}}(\kp^k)|\leq n^2\N\kp^{(1-\frac{2}{n^2+1})k}\log\N\kp$, which follows from \eqref{eqn:LRS_finite} and \eqref{eqn:LRS_2}.  This, along with the trivial estimate 
\[
\sum_{\N\kp\leq x}1\ll [F:\Q]\frac{x}{\log x}
\]
for $x\geq 3$, implies that
\begin{align*}
\sum_{\substack{x<\N\kp^k\leq xe^{1/T} \\ k>16n^2[F:\Q]}}\Lambda_{\pi\times\tilde{\pi}}(\kp^k)&\leq n^2 x^{1-\frac{2}{n^2+1}}\sum_{\substack{\N\kp^k\leq ex \\ k>16n^2[F:\Q]}}\log\N\kp\\
&\ll n^2[F:\Q] x^{1-\frac{2}{n^2+1}+\frac{1}{16n^2[F:\Q]}}\log x\ll n^2[F:\Q]\frac{x}{T}.
\end{align*}
Since $\Lambda_{\pi\times\tilde{\pi}}(\kn)=0$ whenever $\kn$ is not a power of a prime ideal, this concludes our proof.
\end{proof}

\section{Zeroes of \texorpdfstring{$L(s,\pi\times\tilde{\pi})$}{L(s,\83\300{}\80\327{}\83\003{}\83\300{})}}

We require three results on the distribution of zeroes of $L(s,\pi\times\tilde{\pi})$, which are analogous to the key ingredients in Linnik's bound on the least prime in an arithmetic progression: a standard zero-free region with an effective bound on a possible Landau--Siegel zero, a log-free zero density estimate, and a quantification of the Deuring--Heilbronn zero repulsion phenomenon for Landau--Siegel zeroes.

\subsection{A standard zero-free region for \texorpdfstring{$L(s,\pi\times\tilde{\pi})$}{L(s,\83\300{}\80\327{}\83\003{}\83\300{})}}

Let $\pi\in\mathfrak{F}_n$.  In \cite{Humphries}, Humphries proved that if $|\alpha_{j,\pi}(\kp)|\leq 1$ (uniformly in $j$) for all except a density zero subset of prime ideals $\kp$, then there exists a constant $c_{\pi}>0$ dependent on $\pi$ (and hence also on $n$) such that $L(s,\pi\times\tilde{\pi})\neq 0$ in the region
\[
\re(s)\geq 1-\frac{c_{\pi}}{\log(|\im(s)|+e)},\qquad \im(s)\neq 0.
\]
This extended upon work of Goldfeld and Li \cite{Goldfeld_Li}, who proved a weaker zero-free region via a different method under the additional conditions that $F = \Q$ and that $\pi$ is everywhere unramified. Here, we prove an unconditional refinement with improved uniformity in $\pi$.
\begin{theorem}
\label{thm:ZFR}
Let $\pi\in\mathfrak{F}_n$. There exists an absolute constant $\Cl[abcon]{ZFR}>0$ such that the Rankin--Selberg $L$-function $L(s,\pi\times\tilde{\pi})$ is nonvanishing in the region
\[
\re(s)\geq 1-\frac{\Cr{ZFR}}{\log(C(\pi\times\tilde{\pi})(|\im(s)|+e)^{n^2[F:\Q]})}
\]
apart from at most one exceptional zero $\beta_1$.  If $\beta_1$ exists, then it is both real and simple.
\end{theorem}

\begin{proof}
Let $\rho = \beta + i\gamma$ be a zero of $L(s,\pi\times \tilde{\pi})$ with $\beta \geq 1/2$ and $\gamma \neq 0$. We define
\[
\Pi := \pi \boxplus \pi \otimes \left|\det\right|^{i\gamma} \boxplus \pi \otimes \left|\det\right|^{-i\gamma}.
\]
This is an isobaric (noncuspidal) representation of $\GL_{3n}(\A_F)$. The Rankin--Selberg $L$-function $L(s,\Pi\times \tilde{\Pi})$ factorises as
\[
L(s,\pi\times \tilde{\pi})^3 L(s + i\gamma, \pi\times \tilde{\pi})^2 L(s - i\gamma, \pi\times \tilde{\pi})^2 L(s + 2i\gamma, \pi\times \tilde{\pi}) L(s - 2i\gamma, \pi\times \tilde{\pi}).
\]
Since $L(s,\pi\times \tilde{\pi})$ is meromorphic on $\mathbb{C}$ with only a simple pole at $s = 1$, $L(s,\Pi\times \tilde{\Pi})$ is a meromorphic function on $\mathbb{C}$ with a triple pole at $s = 1$, double poles at $s = 1 \pm i\gamma$, and simple poles at $s = 1 \pm 2i\gamma$. Moreover, the functional equation for $L(s,\pi\times \tilde{\pi})$ together with the fact that $\pi\times \tilde{\pi}$ is self-dual (even if $\pi$ itself is not self-dual) implies that if $\rho$ is a zero of $L(s,\pi\times \tilde{\pi})$, then so is $\overline{\rho}$.  Consequently, $s = \beta$ is a zero of $L(s,\Pi\times \tilde{\Pi})$ of order at least $4$.


Define
\[
\Lambda(s,\Pi\times\tilde{\Pi}):=\Lambda(s,\pi\times\tilde{\pi})^3\Lambda(s+i\gamma,\pi\times\tilde{\pi})^2\Lambda(s-i\gamma,\pi\times\tilde{\pi})^2\Lambda(s+2i\gamma,\pi\times\tilde{\pi})\Lambda(s-2i\gamma,\pi\times\tilde{\pi}).
\]
Since $\Lambda(s,\Pi\times\tilde{\Pi})$ is entire of order 1 (regardless of whether $\gamma=0$), it admits a Hadamard factorization
\[
e^{a_{\Pi\times\tilde{\Pi}}+b_{\Pi\times\tilde{\Pi}}s}\prod_{\substack{L(\rho,\Pi\times\tilde{\Pi})=0}}\Big(1-\frac{s}{\rho}\Big)e^{s/\rho}.
\]
where $\rho$ runs through the nontrivial zeroes of $L(s,\Pi\times\tilde{\Pi})$.  A standard calculation shows that
\[
\re(b_{\Pi\times\tilde{\Pi}}) = -\sum_{\rho}\re(\rho^{-1})
\]
(see \cite[Proposition 5.7(3)]{IK}).  Therefore, it follows by equating the real parts of the logarithmic derivatives of $\Lambda(s,\Pi\times\tilde{\Pi})$ and its Hadamard product that for $\sigma>1$, we have
\begin{multline*}
\sum_{\rho\neq0,1}\re\Big(\frac{1}{\sigma-\rho}\Big) \\
= \re\Big(\frac{L'}{L}(\sigma,\Pi\times\tilde{\Pi})\Big)+\frac{3}{\sigma-1}+\frac{3}{\sigma}+\frac{1}{\sigma+2i\gamma-1}+\frac{1}{\sigma+2i\gamma}+\frac{1}{\sigma-2i\gamma-1}\\
+\frac{1}{\sigma-2i\gamma}+\frac{2}{\sigma+i\gamma-1}+\frac{2}{\sigma+i\gamma}+\frac{2}{\sigma-i\gamma-1}+\frac{2}{\sigma-i\gamma}\\
+\frac{1}{2}\log(D_F^{9n^2}\N\kq_{\Pi\times\tilde{\Pi}})+\re\Big(\frac{L'}{L}(\sigma,\Pi_{\infty}\times\tilde{\Pi}_{\infty})\Big).
\end{multline*}
Restricting $\rho$ to the real zeroes $\beta>\frac{1}{2}$ of $L(s,\Pi\times\tilde{\Pi})$ and applying Stirling's formula, we find that there exists an absolute and effectively computable constant $\Cl[abcon]{Stirling}\geq 1$ such that
\begin{multline*}
\sum_{\substack{L(\beta,\Pi\times\tilde{\Pi})=0 \\ \beta>\frac{1}{2}}}\frac{1}{\sigma-\beta} \\
\leq \re\Big(\frac{L'}{L}(\sigma,\Pi\times\tilde{\Pi})\Big)+\frac{3}{\sigma-1}+\frac{2}{\sigma+i\gamma-1}+\frac{2}{\sigma-i\gamma-1}+\frac{1}{\sigma+2i\gamma-1}+\frac{1}{\sigma-2i\gamma-1}\\
+\Cr{Stirling}\log C(\Pi\times\tilde{\Pi}).
\end{multline*}
We now observe that if $\sigma>1$, then
\[
\frac{L'}{L}(\sigma,\Pi\times \tilde{\Pi}) = - \sum_{\kn} \frac{\Lambda_{\pi\times \tilde{\pi}}(\kn)}{\N\kn^{\sigma}} (1 + 2\cos (\gamma \log \N\kn))^2 \leq 0.
\]
This crucially relies on the nonnegativity of $\Lambda_{\pi\times\tilde{\pi}}(\kn)$, even if $\gcd(\kn,\kq_{\pi})\neq\cO_F$ \cite[Lemma a]{Hoffstein}. Therefore, it follows that
\[
\sum_{\substack{L(\beta,\Pi\times\tilde{\Pi})=0 \\ \beta>\frac{1}{2}}}\frac{1}{\sigma-\beta} \leq \frac{3}{\sigma-1}+\frac{4(\sigma-1)}{(\sigma-1)^2+\gamma^2}+\frac{2(\sigma-1)}{(\sigma-1)^2+4\gamma^2}+\Cr{Stirling}\log C(\Pi\times\tilde{\Pi}).
\]
By \eqref{eqn:BH}, we have the bound
\begin{align*}
C(\Pi\times \tilde{\Pi}) & = C(\pi\times \tilde{\pi})^3 C(\pi\times \tilde{\pi},\gamma)^2 C(\pi\times \tilde{\pi},-\gamma)^2 C(\pi\times \tilde{\pi},2\gamma) C(\pi\times \tilde{\pi},-2\gamma)	\\
& \leq C(\pi\times \tilde{\pi})^9 (2|\gamma| + e)^{6n^2 [F:\Q]},
\end{align*}
in which case
\begin{equation}
\label{eqn:starting_point}
\begin{aligned}
\sum_{\substack{L(\beta,\Pi\times\tilde{\Pi})=0 \\ \beta>\frac{1}{2}}}\frac{1}{\sigma-\beta} &\leq \frac{3}{\sigma-1}+\frac{4(\sigma-1)}{(\sigma-1)^2+\gamma^2}+\frac{2(\sigma-1)}{(\sigma-1)^2+4\gamma^2}\\
&+9\Cr{Stirling}\log (C(\pi\times\tilde{\pi}) (2|\gamma|+e)^{n^2[F:\Q]}).
\end{aligned}
\end{equation}

Recall that $L(\beta+i\gamma,\pi\times\tilde{\pi})=0$.  By the discussion at the beginning of the proof, it follows that $\beta$ is a zero of $L(s,\Pi\times\tilde{\Pi})$ with multiplicity at least 4.  Therefore, by \eqref{eqn:starting_point}, the inequality
\begin{equation}
\label{eqn::}
\frac{4}{\sigma-\beta}\leq \frac{3}{\sigma-1}+\frac{4(\sigma-1)}{(\sigma-1)^2+\gamma^2}+\frac{2(\sigma-1)}{(\sigma-1)^2+4\gamma^2}+9\Cr{Stirling}\log (C(\pi\times\tilde{\pi}) (2|\gamma|+e)^{n^2[F:\Q]}).
\end{equation}
holds for all $1<\sigma<\frac{3}{2}$.  However, if $|\gamma|$ is less than a sufficiently small positive multiple of $1/\log C(\pi\times\tilde{\pi})$, then for all $1<\sigma<\frac{3}{2}$, \eqref{eqn::} does not imply a nontrivial upper bound for $\beta$.  By choosing
\[
\sigma = 1+\frac{1}{28\Cr{Stirling}\log (C(\pi\times\tilde{\pi})(2|\gamma|+e)^{n^2[F:\Q]})},
\]
we ensure via \eqref{eqn::} that
\[
\beta\leq 1-\frac{1}{3108\Cr{Stirling}\log(C(\pi\times\tilde{\pi})(2|\gamma|+e)^{n^2[F:\Q]})}
\]
whenever
\[
|\gamma|\geq \frac{1}{7\Cr{Stirling}\log C(\pi\times\tilde{\pi})}.
\]

To handle the case where $0<|\gamma|<1/(7\Cr{Stirling}\log C(\pi\times\tilde{\pi}))$, we proceed as above, but with $L(s,\pi\times\tilde{\pi})$ in place of $L(s,\Pi\times\tilde{\Pi})$.  As described above, if $L(\beta+i\gamma,\pi\times\tilde{\pi})=0$, then $L(\beta-i\gamma,\pi\times\tilde{\pi})=0$.  By analysis that is essentially identical to before, we conclude that if $1<\sigma<\frac{3}{2}$, $L(\beta+i\gamma,\pi\times\tilde{\pi})=0$, $\beta>\frac{1}{2}$, and $0<|\gamma|<1/(\Cr{Stirling}\log C(\pi\times\tilde{\pi}))$, then
\begin{equation}
\label{eqn:low_zeroes}
\sum_{\substack{L(\beta+i\gamma,\pi\times\tilde{\pi})=0 \\ \textup{$\beta+i\gamma$ nontrivial}}}\frac{\sigma-\beta}{(\sigma-\beta)^2+\gamma^2}\leq  \frac{1}{\sigma-1}+\Cr{Stirling}\log C(\pi\times\tilde{\pi}).
\end{equation}
If $\gamma\neq 0$ and $\beta+i\gamma$ is in the sum over zeroes in \eqref{eqn:low_zeroes}, then so is $\beta-i\gamma$.  Therefore, we have
\[
2\frac{\sigma-\beta}{(\sigma-\beta)^2+\gamma^2}\leq  \frac{1}{\sigma-1}+\Cr{Stirling}\log C(\pi\times\tilde{\pi}).
\]
By choosing
\begin{equation}
\label{eqn:sigma_choice}
\sigma = 1+\frac{1}{2\Cr{Stirling}\log C(\pi\times\tilde{\pi})},
\end{equation}
we ensure that
\[
\beta\leq 1-\frac{2\sqrt{1-9\Cr{Stirling}^2\gamma^2(\log C(\pi\times\tilde{\pi}))^2}-1}{6\Cr{Stirling}\log C(\pi\times\tilde{\pi})}.
\]
Our hypothesis that $0<|\gamma|<1/(7\Cr{Stirling}\log C(\pi\times\tilde{\pi}))$ implies that
\begin{equation}
\label{eqn:lowzerobound}
\beta\leq 1-\frac{4\sqrt{10}-7}{42\Cr{Stirling}\log C(\pi\times\tilde{\pi})}\leq 1-\frac{1}{3108\Cr{Stirling}\log(C(\pi\times\tilde{\pi})(|\gamma|+e)^{n^2[F:\Q]})}.
\end{equation}

Finally, we address the real zeroes of $L(s,\pi\times\tilde{\pi})$.  By \eqref{eqn:low_zeroes}, if $1<\sigma<\frac{3}{2}$, then
\begin{equation}
\label{eqn:realzerosum}
\sum_{\substack{L(\beta,\pi\times\tilde{\pi})=0 \\ 1-\frac{1}{3108\Cr{Stirling}\log(C(\pi\times\tilde{\pi})3^{n^2[F:\Q]})}\leq \beta\leq 1}}\frac{1}{\sigma-\beta}\leq \frac{1}{\sigma-1}+\Cr{Stirling}\log C(\pi\times\tilde{\pi}),
\end{equation}
where the zeroes $\beta$ are counted with multiplicity.  Therefore, if $N$ is the total number of zeroes (with multiplicity) in \eqref{eqn:realzerosum}, then
\[
\frac{N}{\sigma-1}\leq \frac{1}{\sigma-1}+\Cr{Stirling}\log C(\pi\times\tilde{\pi}).
\]
With $\sigma$ as in \eqref{eqn:sigma_choice}, it follows that
\[
N\leq \frac{3}{2}.
\]
Since $N$ is an integer, we must have $N\in\{0,1\}$.  We have counted zeroes with multiplicity, so if $N=1$, then the enumerated zero must be simple.
\end{proof}

\subsection{Log-free zero density estimate}
Our log-free zero density estimate is a version of the work of Soundararajan and Thorner \cite[Theorem 1.2]{ST}  in which we explicate the dependence of the implied constant on $n$ and $[F:\Q]$.  Because the details follow those in the proof of \cite[Theorem 1.2]{ST} so closely, we only provide a sketch that outlines our minor adjustments.  These adjustments use \cref{prop:BT,thm:ZFR}.
\begin{proposition}
	\label{prop:LFZDE}
	Let $\pi\in\mathfrak{F}_n$, $T\geq 1$, and $0\leq\sigma\leq 1$.  Define
	\[
	N(\sigma,T):=\#\{\rho=\beta+i\gamma\neq\beta_1\colon \beta\geq\sigma,~|\gamma|\leq T,~L(\rho,\pi\times\tilde{\pi})=0\}.
	\]
	If $\log\log C(\pi\times\tilde{\pi})\gg n^4[F:\Q]^2$ with a sufficiently large implied constant, then
	\[
	N(\sigma,T)\ll n^2[F:\Q](C(\pi\times\tilde{\pi})T^{[F:\Q]})^{10^7 n^2(1-\sigma)}.
	\]
\end{proposition}
\begin{proof}[Sketch of proof]
The following discussion assumes familiarity with the proofs in \cite{ST}.  The analogue of \cite[Lemma 2.3]{ST} over number fields is
	\begin{equation}
	\label{eqn:mertens}
	\sum_{\kn}\frac{\Lambda_{\pi\times\tilde{\pi}}(\kn)}{\N\kn^{1+\eta}}\leq\frac{1}{\eta}+\frac{1}{2}\log C(\pi\times\tilde{\pi})+O(n^2[F:\Q]),
	\end{equation}
	which is easily seen using the shape of $L(s,\pi_{\infty}\times\tilde{\pi}_{\infty})$.  Since $L(s,\pi\times\tilde{\pi})$ has a pole of order 1 at $s=1$, a more careful look at the proof of \cite[Lemma 3.1]{ST} yields the estimates
	\[
	\sum_{\rho}\frac{1+\eta-\beta}{|1+\eta+it-\rho|^2}\leq 2\eta\log(C(\pi\times\tilde{\pi})(2+|t|)^{n^2[F:\Q]})+O(n^2[F:\Q]+\eta^{-1})
	\]
	and
	\begin{equation}
	\label{eqn:local_density_zeroes}
	\#\{\rho\colon |1+it-\rho|\leq\eta\}\leq 10\eta\log(C(\pi\times\tilde{\pi})(2+|t|)^{n^2[F:\Q]})+O(n^2[F:\Q]\eta+1)
	\end{equation}
	for $\eta>0$.  This leads to the choice
	\[
	\frac{\Cr{ZFR}}{\log(C(\pi\times\tilde{\pi})T^{n^2[F:\Q]})}<\eta\leq \frac{1}{200n^2[F:\Q]}
	\]
	when replicating the arguments in \cite[Section 4]{ST} and the choice
	\[
	K>2000\eta\log(C(\pi\times\tilde{\pi})(2+|t|)^{n^2[F:\Q]})+O(n^2[F:\Q]\eta+1)
	\]
	in \cite[Lemma 4.2]{ST}.  In order to replace the use of \cite[Theorem 2.4]{ST} by \cref{prop:BT} in the estimation of \cite[Equation 4.6]{ST}, one chooses
	\[
	K = 4800n^2\eta\log(C(\pi\times\tilde{\pi})(2+|t|)^{[F:\Q]})+O(n^2[F:\Q]\eta+1).
	\]
	\cref{prop:LFZDE} follows for $\sigma\leq 1-\frac{\Cr{ZFR}}{2\log(C(\pi\times\tilde{\pi})T^{n^2[F:\Q]})}$ from inserting these changes into the proof of \cite[Theorem 1.2]{ST}.  For $\sigma$ in the complementary range, \cref{thm:ZFR} ensures that $N(\sigma,T)\leq 1$.
\end{proof}

\subsection{Zero repulsion}

If $L(s,\pi\times\tilde{\pi})$ has a Landau--Siegel zero especially close to $s=1$, then the standard zero-free region in \cref{thm:ZFR} improves noticeably.

\begin{proposition}
\label{prop:Moreno}
	Let $\pi\in\mathfrak{F}_n$.  If the exceptional zero $\beta_1$ in \cref{thm:ZFR} exists, then here exist absolute and effectively computable constants $\Cl[abcon]{Moreno1},\Cl[abcon]{Moreno2}>0$ such that apart from the point $s=\beta_1$, $L(s,\pi\times\tilde{\pi})$ is nonzero in the region
	\[
	\re(s)\geq  1-\Cr{Moreno2}\frac{\log\Big(\cfrac{\Cr{Moreno1}}{(1-\beta_1)\log (C(\pi\times\tilde{\pi})(|\im(s)|+e)^{n^2[F:\Q]})}\Big)}{\log (C(\pi\times\tilde{\pi})(|\im(s)|+e)^{n^2[F:\Q]})}.
	\]
\end{proposition}
\begin{remark}
A similar result appears in \cite[Theorem 4.2]{Moreno}, but with a weaker notion of the analytic conductor.  We provide a self-contained proof for the sake of completeness.
\end{remark}

\begin{proof}
	Assume that the exceptional zero $\beta_1>0$ in \cref{thm:ZFR} exists.  We follow the ideas in \cite[Theorem 5.1]{LMO} as applied to the Dedekind zeta function $\zeta_F(s)$.  The nonnegativity of the Dirichlet coefficients $\Lambda_{\pi\times\tilde{\pi}}(\kn)$ will play a key role.
	
	Since $(s-1)L(s,\pi\times\tilde{\pi})$ is entire of order $1$, it has the Hadamard product representation
	\[
	(s-1)L(s,\pi\times\tilde{\pi})=s^r e^{\alpha_1+\alpha_2 s}\prod_{\omega\neq 0}\Big(1-\frac{s}{\omega}\Big)e^{s/\omega},
	\]
	where $r \geq 1$ is the order of the zero of $L(s,\pi\times \tilde{\pi})$ at $s = 0$ and the product runs through {\it all} non-zero roots $\omega$ of $L(s,\pi\times\tilde{\pi})$ (trivial and nontrivial).  In the region $\re(s)>1$, we take the $(2j-1)$-th derivative of $-\frac{L'}{L}(s,\pi\times\tilde{\pi})$, and we represent it in two ways---as a sum over zeroes using the Hadamard product, and as a Dirichlet series using \eqref{eqn:L'/L}. Thus, for $\re(s)>1$, we arrive at
	\begin{equation}
	\label{eqn:high-deriv1}
	\frac{1}{(2j-1)!}\sum_{\kn}\frac{\Lambda_{\pi\times\tilde{\pi}}(\kn)}{\N\kn^s}(\log\N\kn)^{2j-1}=\frac{1}{(s-1)^{2j}} -\sum_{\omega}\frac{1}{(s-\omega)^{2j}},
	\end{equation}
	where the sum is over all zeroes including $\omega = 0$.
	
	Let $\beta'+i\gamma'\neq\beta_1$ be a zero of $L(s,\pi\times\tilde{\pi})$ (trivial or nontrivial).  By considering \eqref{eqn:high-deriv1} at $s=2$ and $s=2+i\gamma'$, we conclude that
	\begin{equation}
	\label{eqn:high-deriv2}
	\begin{aligned}
	&\frac{1}{(2j-1)!}\sum_{\kn}\Lambda_{\pi\times\tilde{\pi}}(\kn)(\log\N\kn)^{2j-1}\N\kp^{-2m}(1+(\N\kp^m)^{-i\gamma'})\\
	&=1+\frac{1}{(1+i\gamma')^{2j}}-\frac{1}{(2-\beta_1)^{2j}}-\frac{1}{(2-\beta_1+i\gamma')^{2j}}-\sum_{n=1}^{\infty}z_n^j.
	\end{aligned}
	\end{equation}
	We sum \eqref{eqn:high-deriv2} at $s=2+i\gamma'$ with \eqref{eqn:high-deriv2} at $s = 2$, take real parts, and deduce via the nonnegativity of $\Lambda_{\pi\times\tilde{\pi}}(\kn)$ that
	\begin{equation}
		\label{eqn:high_deriv}
		\begin{aligned}
			0&\leq \frac{1}{(2j-1)!}\sum_{\kn}\frac{\Lambda_{\pi\times\tilde{\pi}}(\kn)(\log\N\kn)^{2j-1}}{\N\kn^2}(1+\cos(\gamma' \log \N\kn))\\
	&=1 - \frac{1}{(2-\beta_1)^{2j}}+\re\Big(\frac{1}{(1+i\gamma')^{2j}}-\frac{1}{(2-\beta_1+i\gamma')^{2j}}\Big)-\sum_{n=1}^{\infty}\re(z_n^j),	
		\end{aligned}
	\end{equation}
	where the set $\{z_n\colon n\geq 1\}$ is equal to $\{(2-\omega)^{-2}, (2+i\gamma'-\omega)^{-2} : \omega \neq \beta_1\}$ with the labeling chosen such that $|z_1|\geq |z_n|$ for all $n\geq 1$; in particular, since there exists some $n$ for which $z_n = 2 + i\gamma' - \rho' = 2 - \beta'$, we must have that $|z_1|\geq (2-\beta')^{-2}$. This lower bound on $|z_1|$ applied to \eqref{eqn:high_deriv} and a Taylor expansion imply that there exists a constant $\Cl[abcon]{M1}\geq 1$ such that
	\begin{equation}
	\label{eqn:turan_upper}
	\sum_{n=1}^{\infty}\re(z_n^j)\leq 1-\frac{1}{(2-\beta_1)^{2j}}+\re\Big(\frac{1}{(1+i\gamma')^{2j}}-\frac{1}{(2-\beta_1+i\gamma')^{2j}}\Big)\leq \Cr{M1} j(1-\beta_1).
	\end{equation}
	
	We need a lower bound for the left hand side of \eqref{eqn:turan_upper}.  Define
	\[
	L:=|z_1|^{-1} \sum_{n=1}^{\infty}|z_n|.
	\]
	For $t\in\R$, the number of nontrivial zeroes $\rho=\beta+i\gamma$ of $L(s,\pi\times\tilde{\pi})$ satisfying $|\gamma-t|\leq 1$ is $O(\log C(\pi\times\tilde{\pi},t))$.  Thus by \eqref{eqn:BH}, there exist constants $\Cl[abcon]{M2}\geq 1$ and $\Cl[abcon]{M3}\geq 1$ such that
	\begin{equation}
	\label{eqn:L_bound_upper}
	L\leq\Cr{M2} (2-\beta')^2\sum_{\omega}\Big(\frac{1}{|2-\omega|^2}+\frac{1}{|2+i\gamma'-\omega|^2}\Big)\leq \Cr{M3} \log(C(\pi\times\tilde{\pi})(e+|\gamma'|)^{n^2[F:\Q]}).
	\end{equation}
	The lower bound for power sums in \cite[Theorem 4.2]{LMO} implies that there exists an integer $1\leq j_1\leq 24L$ such that
	\begin{equation}
	\label{eqn:turan_lower}
	8 \sum_{n=1}^{\infty}\re(z_n^{j_1})\geq (2-\beta')^{-2j_1}\geq \exp(-2j_1(1-\beta')).
	\end{equation}
	We combine \eqref{eqn:turan_upper} and \eqref{eqn:turan_lower} to obtain the bound $\exp(-2j_1(1-\beta'))\leq 8\Cr{M1} j_1(1-\beta_1)$.  Since $j_1\leq 24 L\leq 24\Cr{M3} \log(C(\pi\times\tilde{\pi})(e+|\gamma'|)^{n^2[F:\Q]})$ per \eqref{eqn:L_bound_upper}, we arrive at the bound
	\[
	\exp(-48\Cr{M3} \log(C(\pi\times\tilde{\pi})(e+|\gamma'|)^{n^2[F:\Q]})(1-\beta'))\leq 192\Cr{M1}\Cr{M3} \log(C(\pi\times\tilde{\pi})(e+|\gamma'|)^{n^2[F:\Q]})(1-\beta_1).
	\]
	We obtain the desired result by solving the above inequality for $\beta'$.
\end{proof}

We use \cref{prop:Moreno} to determine an upper bound for $\beta_1$ (if it exists).

\begin{corollary}
	\label{cor:Siegel}
	There exists an absolute and effectively computable constant $\Cl[abcon]{Siegel}>0$ such that if the exceptional zero $\beta_1$ in \cref{thm:ZFR} exists, then $\beta_1\leq 1-C(\pi\times\tilde{\pi})^{-\Cr{Siegel}}$.
\end{corollary}
\begin{proof}
	If $\beta'\neq \beta_1$ is a real zero of $L(s,\pi\times\tilde{\pi})$ (trivial or nontrivial), then by \cref{prop:Moreno},
	\begin{equation}
	\label{eqn:moreno}
	\begin{aligned}
	\beta'&\leq 1-\frac{\Cr{Moreno2}\log\frac{1}{1-\beta_1}+\Cr{Moreno2}\log\Cr{Moreno1}-\Cr{Moreno2}\log\log(C(\pi\times\tilde{\pi})e^{n^2[F:\Q]})}{\log(C(\pi\times\tilde{\pi})e^{n^2[F:\Q]})}.
	\end{aligned}
	\end{equation}
First, consider the case where $C(\pi\times\tilde{\pi})e^{n^2[F:\Q]}$ is large enough so that
	\begin{equation}
	\label{eqn:bound_1/2}
	-\frac{\Cr{Moreno2}\log\Cr{Moreno1}-\Cr{Moreno2}\log\log(C(\pi\times\tilde{\pi})e^{n^2[F:\Q]})}{\log(C(\pi\times\tilde{\pi})e^{n^2[F:\Q]})}\leq \frac{1}{2}.
	\end{equation}
	This is satisfied when $C(\pi\times\tilde{\pi})e^{n^2[F:\Q]}$ is larger than a certain absolute and effectively computable constant.  By \eqref{eqn:moreno}, we have
	\[
	\beta'\leq \frac{1}{2}-\Cr{Moreno2}\frac{\log\frac{1}{1-\beta_1}}{\log(C(\pi\times\tilde{\pi})e^{n^2[F:\Q]})}.
	\]
	Define $c$ so that
	\[
	\beta_1=1-(C(\pi\times\tilde{\pi})e^{n^2[F:\Q]})^{-c}.
	\]
	It follows that $\beta'\leq \frac{1}{2}-\Cr{Moreno2}c$.  If $c\geq 7/(2\Cr{Moreno2})$, then we have determined that $L(\sigma,\pi\times\tilde{\pi})\neq 0$ for all $\sigma> -3$.  This contradicts the fact that the trivial zeroes of $\zeta_F(s)$ are included among the trivial zeroes of $L(s,\pi\times\tilde{\pi})$, and $\zeta_F(s)$ has a trivial zero in the set $\{-2,-1,0\}$.  We conclude that $\beta_1\leq 1-(C(\pi\times\tilde{\pi})e^{n^2[F:\Q]})^{-7/(2\Cr{Moreno2})}$.  	The Minkowski bound $[F:\Q]\ll \log D_F$ yields $\beta_1\leq 1-C(\pi\times\tilde{\pi})^{-\Cr{Siegel}}$ once $\Cr{Siegel}$ is sufficiently large.
	
	If $C(\pi\times\tilde{\pi})e^{n^2[F:\Q]}$ is not large enough to satisfy \eqref{eqn:bound_1/2}, then there exists an absolute and effectively computable constant $0<B<1$ such that any real zero $\beta$ of $L(s,\pi\times\tilde{\pi})$ satisfies
	\[
	\beta\leq 1-B=1-C(\pi\times\tilde{\pi})^{-\frac{\log B^{-1}}{\log C(\pi\times\tilde{\pi})}}\leq 1-C(\pi\times\tilde{\pi})^{-\Cr{Siegel}}
	\]
	once $\Cr{Siegel}$ is made sufficiently large.
\end{proof}

\section{Proof of \texorpdfstring{\cref{thm:main_theorem}}{Theorem \ref*{thm:main_theorem}}}

We prove the result when the exceptional zero $\beta_1$ in \cref{thm:ZFR} exists; the complementary case is easier.  In order to apply \cref{prop:BT}, we make the initial restrictions
\begin{equation}
\label{eqn:ranges_prelim}
\log\log C(\pi\times\tilde{\pi})\gg n^4[F:\Q]^2,\quad x\geq C(\pi\times\tilde{\pi})^{32n^2},\quad 1\leq T\leq x^{\frac{1}{16n^2[F:\Q]}}.
\end{equation}
We begin with the version of Perron's integral formula proved in \cite[Corollary 2.2]{LY} applied to $L(s,\pi\times\tilde{\pi})$:
\begin{multline*}
\sum_{\N\kn\leq x}\Lambda_{\pi\times\tilde{\pi}}(\kn)=\frac{1}{2\pi i}\int_{1+\frac{1}{\log x}-iT}^{1+\frac{1}{\log x}+iT}-\frac{L'}{L}(s,\pi\times\tilde{\pi})\frac{x^s}{s}\,ds	\\
+O\Big(\sum_{|\N\kn-x| \leq \frac{x}{\sqrt{T}}}\Lambda_{\pi\times\tilde{\pi}}(\kn)+\frac{x}{\sqrt{T}}\sum_{\kn}\frac{\Lambda_{\pi\times\tilde{\pi}}(\kn)}{\N\kn^{1+\frac{1}{\log x}}}\Big).
\end{multline*}
One constructs a contour given by the perimeter of the rectangle with vertices $1+\frac{1}{\log x}\pm iT$ and $-U\pm iT$, where $U\geq 1$ is fixed, arbitrarily large, and chosen so that no trivial zero of $L(s,\pi\times\tilde{\pi})$ lies within a distance of $\frac{1}{4n^2}$ from the line $\re(s)=-U$.  The residues in the interior of this rectangle arise from the trivial zeroes (which are handled by \eqref{eqn:LRS_2}), the nontrivial zeroes $\rho=\beta+i\gamma$ with $\gamma\in\R$ and $0<\beta<1$, and the residue at $s=0$.

The bound \eqref{eqn:local_density_zeroes} and Stirling's formula, when applied to the logarithmic derivative of the Hadamard product for $L(s,\pi\times\tilde{\pi})$, imply that if $-\frac{1}{2}\leq\re(s)\leq 2$, then
\begin{multline*}
-\frac{L'}{L}(s,\pi\times\tilde{\pi}) = \frac{1}{s}+\frac{1}{s-1}-\sum_{|s-\rho|<1}\frac{1}{s-\rho}\\
-\sum_{|s+\mu_{j,j',\pi\times\tilde{\pi}}(v)|<1}\frac{1}{s+\mu_{j,j',\pi\times\tilde{\pi}}(v)}+O(\log C(\pi\times\tilde{\pi},|\im(s)|)).
\end{multline*}
(See also \cite[Proposition 5.7]{IK}.)  This can be extended to a wider strip using the functional equation for $L(s,\pi\times\tilde{\pi})$, leading to a strong bound on both the residue at $s=0$ and the legs of the contour to the left of $\re(s) = 1+\frac{1}{\log x}$.  Ultimately, with the help of \eqref{eqn:ranges_prelim}, we arrive at
\begin{multline*}
	\sum_{\N\kn\leq x}\Lambda_{\pi\times\tilde{\pi}}(\kn)=x-\sum_{\substack{0 < \beta < 1 \\ |\gamma|\leq T}} \frac{x^{\rho}}{\rho}+O\Big(\sum_{x - \frac{x}{\sqrt{T}} \leq \N\kn \leq x + \frac{x}{\sqrt{T}}} \Lambda_{\pi\times\tilde{\pi}}(\kn)+\frac{x}{\sqrt{T}}\sum_{\kn}\frac{\Lambda_{\pi\times\tilde{\pi}}(\kn)}{\N\kn^{1+\frac{1}{\log x}}}+\frac{x(\log  x)^2}{T}\Big).
\end{multline*}
(See also \cite[Sections 4 and 5]{LY}.)  The bounds \eqref{eqn:mertens} and \eqref{eqn:ranges_prelim} imply that
\begin{align*}
\sum_{\N\kn\leq x}\Lambda_{\pi\times\tilde{\pi}}(\kn)=x-\sum_{\substack{0 < \beta < 1 \\ |\gamma|\leq T}} \frac{x^{\rho}}{\rho}+O\Big(\sum_{x - \frac{x}{\sqrt{T}} \leq \N\kn \leq x + \frac{x}{\sqrt{T}}}\Lambda_{\pi\times\tilde{\pi}}(\kn)+\frac{x(\log x)^2}{\sqrt{T}}\Big).
\end{align*}
By \cref{prop:BT} and \eqref{eqn:ranges_prelim}, we have
\begin{equation}
\label{eqn:explicit_formula}
\sum_{\N\kn\leq x}\Lambda_{\pi\times\tilde{\pi}}(\kn)=x-\sum_{\substack{0 < \beta < 1 \\ |\gamma|\leq T}}\frac{x^{\rho}}{\rho}+O\Big(\frac{x(\log x)^2}{\sqrt{T}}\Big).
\end{equation}

It follows from \eqref{eqn:explicit_formula} that for $2\leq h\leq x$, we have that
\[
\sum_{x<\N\kn\leq x+h}\Lambda_{\pi\times\tilde{\pi}}(\kn)=h-\sum_{\substack{0 < \beta < 1 \\ |\gamma|\leq T}} \frac{(x+h)^{\rho}-x^{\rho}}{\rho}+O\Big(\frac{x(\log x)^2}{\sqrt{T}}\Big).
\]
We observe that
\[
\Big|\frac{(x+h)^{\rho}-x^{\rho}}{\rho}\Big|\leq\min\Big\{hx^{\beta-1},\frac{3x^{\beta}}{|\gamma|}\Big\}
\]
using the bounds
\[
\Big|\frac{(x+h)^{\rho}-x^{\rho}}{\rho}\Big|=\Big|\int_x^{x+h}\tau^{\rho-1}\,d\tau\Big|\leq \int_x^{x+h}\tau^{\beta-1}\,d\tau\leq  hx^{\beta-1}
\]
and
\[
\Big|\frac{(x+h)^{\rho}-x^{\rho}}{\rho}\Big|\leq \frac{(2x)^{\beta}+x^{\beta}}{|\gamma|}\leq \frac{3x^{\beta}}{|\gamma|}.
\]
Finally, note that by the mean value theorem, there exists $\xi\in[x,x+h]$ such that
\[
\frac{(x+h)^{\beta_1}-x^{\beta_1}}{\beta_1}=h\xi^{\beta_1-1}.
\]
Thus, there exists $\xi\in[x,x+h]$ such that
\begin{equation}
\label{eqn:pre_zero_count_sum}
\begin{aligned}
\Big|\frac{1}{h}\sum_{x<\N\kn\leq x+h}\Lambda_{\pi\times\tilde{\pi}}(\kn) - 1 + \xi^{\beta_1-1}\Big|\ll \sideset{}{'}\sum_{\substack{0<\beta<1 \\ |\gamma|\leq \frac{x\log x}{h} }} x^{\beta-1}+\frac{x}{h}\sideset{}{'}\sum_{\substack{0<\beta<1 \\ \frac{x\log x}{h}<|\gamma|\leq T }}\frac{x^{\beta-1}}{|\gamma|}+\frac{x(\log x)^2}{h\sqrt{T}},
\end{aligned}
\end{equation}
where $\sum'$ denotes a sum over nontrivial zeroes $\rho=\beta+i\gamma\neq\beta_1$.  Next, we subdivide the zeroes into $O(\log x)$ dyadic intervals $e^j<|\gamma|\leq e^{j+1}$ and deduce that the right hand side of \eqref{eqn:pre_zero_count_sum} is
\begin{equation}
\label{eqn:sum_zeroes}
\ll \sideset{}{'}\sum_{\substack{ |\gamma|\leq \frac{x\log x}{h} \\ 0<\beta<1}}x^{\beta-1}+\frac{x\log x}{h}\sup_{\frac{x\log x}{h}\leq M\leq T}\frac{1}{M}\sideset{}{'}\sum_{\substack{ |\gamma| \leq M \\ 0<\beta<1}}x^{\beta-1}+\frac{x(\log x)^2}{h\sqrt{T}}.
\end{equation}

At this stage, we require some constraints on the relationships amongst $x$, $T$, $h$, and $C(\pi\times\tilde{\pi})$.  To help simplify future calculations, we choose $A\geq 10^7$ (the constant in the exponent in \cref{prop:LFZDE}) and
\begin{equation}
\label{eqn:choices}
C(\pi\times\tilde{\pi})\leq x^{\theta},\quad x^{1-\delta}\leq h\leq x,\quad T=x^{\frac{1}{4An^2[F:\Q]}},\quad \theta,\delta[F:\Q]\in\Big[0,\frac{1}{16An^2}\Big].
\end{equation}
We will take $A$ to be sufficiently large later on.  Subject to \eqref{eqn:choices}, if $\beta+i\gamma$ is a zero other than $\beta_1$, then by \cref{thm:ZFR}, we have that
\begin{equation}
\label{eqn:zfr_x}
\beta\leq 1-2\Cr{ZFR}A/\log x,\qquad |\gamma|\leq T=x^{\frac{1}{4An^2[F:\Q]}}.
\end{equation}
By \cref{prop:Moreno}, we also have that
\begin{equation}
\label{eqn:dh_x}
\beta\leq 1-2\Cr{Moreno2}A\log\Big(\frac{2A\Cr{Moreno1}}{(1-\beta_1)\log x}\Big)/\log x,\qquad|\gamma|\leq T=x^{\frac{1}{4An^2[F:\Q]}}.
\end{equation}
Finally, by \cref{prop:LFZDE}, we have that
\begin{equation}
\label{eqn:lfzde_x}
N(\sigma,M)\ll n^2[F:\Q]x^{\frac{1-\sigma}{2}},\qquad M\leq T=x^{\frac{1}{4An^2[F:\Q]}}.
\end{equation}
Now, if $M\leq T$ and
\[
\eta=\max\Big\{\Cr{ZFR},\Cr{Moreno2}\log\Big(\frac{2A\Cr{Moreno1}}{(1-\beta_{1})\log x}\Big)\Big\},
\]
then by \eqref{eqn:zfr_x}, \eqref{eqn:dh_x}, and \eqref{eqn:lfzde_x}, we have that
\begin{align}
\label{eqn:sum'}
\sideset{}{'}\sum_{\substack{ |\gamma| \leq M \\ 0<\beta<1}}x^{\beta-1}\ll \log x\int_0^{1-\frac{2A\eta}{\log x}}N(\sigma,M)x^{\sigma-1}d\sigma &\ll n^2[F:\Q]\log x\int_0^{1-\frac{2A\eta}{\log x}}x^{\frac{\sigma-1}{2}}d\sigma\notag\\
&\ll n^2[F:\Q]e^{-A\eta}\notag\\
&=n^2[F:\Q]\min\Big\{e^{-\Cr{ZFR}A},\Big(\frac{(1-\beta_1)\log x}{2A\Cr{Moreno1}}\Big)^{\Cr{Moreno2}A}\Big\}\\
&\ll n^2[F:\Q]e^{-\Cr{ZFR}A}\min\{1,((1-\beta_1)\log x)^{\Cr{Moreno2}A}\}\notag.
\end{align}
Since $\min\{1,a^b\}\leq \min\{1,a\}$ for $a>0$ and $b\geq 1$, it follows that if $A>1/\Cr{Moreno2}$, then \eqref{eqn:sum'} is
\[
\ll n^2[F:\Q]e^{-\Cr{ZFR}A}\min\{1,(1-\beta_1)\log x\}.
\]
We apply this bound to the two sums in \eqref{eqn:sum_zeroes} and invoke \eqref{eqn:pre_zero_count_sum} and \eqref{eqn:choices} to deduce the bound
\begin{equation}
\label{after}
\begin{aligned}
&\Big|\frac{1}{h}\sum_{x<\N\kn\leq x+h}\Lambda_{\pi\times\tilde{\pi}}(\kn) - 1 + \xi^{\beta_1-1}\Big|\\
&\ll n^2[F:\Q]e^{-\Cr{ZFR}A}\min\{1,(1-\beta_1)\log x\}+x^{-\frac{1}{16An^2[F:\Q]}}(\log x)^2.
\end{aligned}
\end{equation}
If $x\geq (An^2[F:\Q])^{900 A n^2[F:\Q]}$, then
\[
\log x\leq x^{\frac{1}{544An^2[F:\Q]}},
\]
and \eqref{after} is
\begin{equation}
\label{eqn:after2}
\ll n^2[F:\Q]e^{-\Cr{ZFR}A}\min\{1,(1-\beta_1)\log x\}+x^{-\frac{1}{17An^2[F:\Q]}}.
\end{equation}
If we also require that $x\geq (e^{\Cr{ZFR}A}C(\pi\times\tilde{\pi})^{\Cr{Siegel}})^{17An^2[F:\Q]}$, then by \cref{cor:Siegel}, \eqref{eqn:after2} is
\[
\ll n^2[F:\Q]e^{-\Cr{ZFR}A}\min\{1,(1-\beta_1)\log x\}\ll n^2[F:\Q]e^{-\Cr{ZFR}A}\min\{1,(1-\beta_1)\log \xi\}.
\]
The bound $\min\{1,(1-\beta_1)\log \xi\}\ll (1-\xi^{\beta_1-1})$ is easily demonstrated.  \cref{thm:main_theorem} follows once we invoke \eqref{eqn:BH} to bound $C(\pi\times\tilde{\pi})$.

\section{Proof of \cref{hypH}}

In light of \cref{thm:main_theorem} and the nonnegativity of $\Lambda_{\pi\times\tilde{\pi}}(\kn)$, it suffices to prove that if $n\in\{1,2,3,4\}$, $\pi\in\mathfrak{F}_n$, and $x>C(\pi)$, then
\begin{equation}
\label{eqn:prime_powers_bounded}
\sum_{\substack{\N\kn\in[x,2x] \\ \textup{$\kn$ composite}}}\Lambda_{\pi\times\tilde{\pi}}(\kn)\ll n^2 x^{1-\frac{1}{2(n^2+1)}}(\log x)^3.
\end{equation}
Our main tools for this estimate are existing bounds towards GRC, namely \eqref{eqn:LRS_finite} and \eqref{eqn:LRS_2}, and the following result of Brumley.

\begin{lemma}
\label{lem:second_order}
	Let $n\in\{1,2,3,4\}$ and $\pi\in\mathfrak{F}_n$.  If $\epsilon>0$, then
\begin{equation}
\label{eqn:secondorder}
\prod_{\kp}\sum_{r=0}^{\infty}\frac{\max_{1\leq j\leq n}|\alpha_{j,\pi}(\kp)|^{2r}}{\N\kp^{r(1+\epsilon)}}\ll_{n,\epsilon}C(\pi)^{\epsilon}.
\end{equation}
\end{lemma}
\begin{proof}
	See \cite[Theorem 1]{Brumley_2}; the Dirichlet series to be bounded is denoted by $L(s,\pi,\left|\max\right|^2)$ therein.  For $n\in\{1,2,3\}$, the proof relies on basic properties of $L(s,\pi\times\tilde{\pi})$.  When $n=4$, the proof relies on the automorphy of the exterior square lift from $\GL_4$ to $\GL_6$, as proved by Kim \cite{Kim}.
\end{proof}

We begin with the observation that \eqref{eqn:prime_powers_bounded} is
\begin{equation}
\label{eqn:Rankin_trick}
\leq x\sum_{\substack{x<\N\kn\leq 2x \\ \textup{$\kn$ composite}}}\frac{\Lambda_{\pi\times\tilde{\pi}}(\kn)}{\N\kn}.
\end{equation}
Note that there are $O(\log x)$ ramified prime ideals, each of which has norm at most $x$.  At each such prime ideal, we have the bound
\[
\Lambda_{\pi\times\tilde{\pi}}(\kn)\leq n^2\N\kn^{1-\frac{2}{n^2+1}}\Lambda(\kn).
\]
Therefore, the contribution from the ramified primes to \eqref{eqn:Rankin_trick} is
\begin{multline*}
\ll n^2 x(\log x)\sum_{2\leq r\leq \frac{\log(2x)}{\log 2}}~\sum_{\substack{x^{1/r}\leq \N\kp\leq (2x)^{1/r} \\ \gcd(\kp,\kq_{\pi})\neq\cO_F}}\N\kp^{-r\frac{2}{n^2+1}}\\
\ll n^2(\log x)^2  x \sum_{2\leq r\leq \frac{\log(2x)}{\log 2}}x^{-\frac{2}{n^2+1}}\ll n^2 x^{1-\frac{2}{n^2+1}}(\log x)^3.
\end{multline*}

For integers $r\geq 1$, it follows from the definition of $\Lambda_{\pi\times\tilde{\pi}}(\kp^r)$ that if $\gcd(\kp,\kq_{\pi})=\cO_F$, then
\[
\Lambda_{\pi\times\tilde{\pi}}(\kp^r)\leq n^2\max_{1\leq j\leq n}|\alpha_{j,\pi}(\kp)|^{2r}\log\N\kp.
\]
Define
\[
\beta_{\kp}:=\N\kp^{-1}\max_{1\leq j\leq n}|\alpha_{j,\pi}(\kp)|^2.
\]
By \eqref{eqn:LRS_finite}, we have that $\beta_{\kp}\leq \N\kp^{-2/(n^2+1)}$.  The contribution from composite $\kn$ with $\gcd(\kn,\kq_{\pi})=\cO_F$ to \eqref{eqn:Rankin_trick} is
\begin{equation}
\label{eqn:prime_powers_first_steps}
\begin{aligned}
\ll n^2 x (\log x)\sum_{r=2}^{\infty}\sum_{x^{1/r}<\N\kp\leq (2x)^{1/r}}\beta_{\kp}^r &\ll n^2x  (\log x)\sum_{2\leq R\leq \frac{\log(2x)}{\log 2}}~\sum_{x^{1/R}<\N\kp\leq (2x)^{1/R}}~\sum_{r=R}^{\infty}\beta_{\kp}^r \\
&=n^2 x (\log x)\sum_{2\leq R\leq \frac{\log(2x)}{\log 2}}~\sum_{x^{1/R}<\N\kp\leq (2x)^{1/R}}\frac{\beta_{\kp}^{R}}{1-\beta_{\kp}}\\
&\leq n^2 x^{1-\frac{1}{n^2+1}}(\log x)\sum_{2\leq R\leq \frac{\log(2x)}{\log 2}}~\sum_{x^{1/R}<\N\kp\leq (2x)^{1/R}}\frac{\beta_{\kp}}{1-\beta_{\kp}}.
\end{aligned}\hspace{-1.5mm}
\end{equation}
In \eqref{eqn:prime_powers_first_steps}, the contribution from prime ideals with norm at most $2^{n^2}$ is $O(1)$ (since $n\leq 4$).  If $\N\kp>2^{n^2}$, then $1-\beta_{\kp}\geq\frac{1}{2}$.  Thus, \eqref{eqn:prime_powers_first_steps} is
\begin{equation}
\label{eqn:prime_powers_next_step}
\begin{aligned}
&\leq n^2 x^{1-\frac{1}{n^2+1}}(\log x)\sum_{2\leq R\leq \frac{\log(2x)}{\log 2}}\Big(2\sum_{x^{1/R}<\N\kp\leq (2x)^{1/R}}\beta_{\kp}+O(1)\Big)\\
&\leq 2n^2 x^{1-\frac{1}{2(n^2+1)}}(\log x)\sum_{2\leq R\leq \frac{\log(2x)}{\log 2}}\Big(2\sum_{x^{1/R}<\N\kp\leq (2x)^{1/R}}\beta_{\kp}\N\kp^{-\frac{1}{n^2+1}}+O(1)\Big).
\end{aligned}
\end{equation}
Note that $x\leq 2\log(x+1)$ for all $0\leq x\leq 5/2$.  Since $0\leq \beta_{\kp}\N\kp^{-\frac{1}{n^2+1}}\leq 1$ for all $\kp$, it follows that \eqref{eqn:prime_powers_next_step} is
\begin{equation}
\label{eqn:prime_powers_last_step}
\begin{aligned}
&\leq 2n^2 x^{1-\frac{1}{2(n^2+1)}}(\log x)\sum_{2\leq R\leq \frac{\log(2x)}{\log 2}}\Big(2\sum_{x^{1/R}<\N\kp\leq (2x)^{1/R}}\log(1+\beta_{\kp}\N\kp^{-\frac{1}{n^2+1}})+O(1)\Big)\\
&\ll n^2x^{1-\frac{1}{2(n^2+1)}}(\log x)\sum_{\kp}\log\Big(\sum_{r=0}^{\infty}\beta_{\kp}\N\kp^{-\frac{r}{n^2+1}}\Big)\\
&=n^2x^{1-\frac{1}{2(n^2+1)}}(\log x)\log\Big(\prod_{\kp}\sum_{r=0}^{\infty}\frac{\max_{1\leq j\leq n}|\alpha_{j,\pi}(\kp)|^{2r}}{\N\kp^{r(1+\frac{1}{n^2+1})}}\Big)
\end{aligned}
\end{equation}
By \cref{lem:second_order} and our hypothesis that $x\geq C(\pi)$, \eqref{eqn:prime_powers_last_step} is $\ll x^{1-\frac{1}{2(n^2+1)}}(\log x)^2$.

\bibliographystyle{abbrv}
\bibliography{GeneralizedLinnik}

\end{document}